\documentclass[a4paper]{cas-sc}

\usepackage[numbers]{natbib}
\usepackage{subcaption}
\def\tsc#1{\csdef{#1}{\textsc{\lowercase{#1}}\xspace}}
\tsc{WGM}
\tsc{QE}

\newtheorem{theorem}{Theorem}
\newtheorem{lemma}[theorem]{Lemma}
\newdefinition{rmk}{Remark}
\newproof{pf}{Proof}
\newproof{pot}{Proof of Theorem \ref{thm}}
\newtheorem{corollary}[theorem]{Corollary}
\newtheorem{proposition}[theorem]{Proposition}
\newtheorem{definition}[theorem]{Definition}
\newtheorem{example}{Example}

\newcommand{\C}{\mathbb C}
\newcommand{\R}{\mathbb R}

\newcommand{\A}{\mathcal A}

\newcommand{\Rr}{\mathcal R}
\newcommand{\Ll}{\mathcal L}

\newcommand{\Ss}{\mathcal{S}}

\renewcommand{\leq}{\leqslant}
\renewcommand{\geq}{\geqslant}

\DeclareMathOperator{\dom}{dom}
\DeclareMathOperator{\mul}{mul}
\DeclareMathOperator{\ran}{ran}

\DeclareMathOperator{\ap}{ap}

\DeclareMathOperator{\Span}{span}

\DeclareMathOperator{\diag}{diag}

\begin{document}
\let\WriteBookmarks\relax
\def\floatpagepagefraction{1}
\def\textpagefraction{.001}

\shorttitle{Block numerical ranges for linear relations}    

\shortauthors{}  

\title [mode = title]{Block numerical ranges for linear relations}  

\tnotemark[1] 

\tnotetext[1]{} 

%

\author[1]{Rabeb~Aydi}


\ead{rabeb.aydi.2019@gmail.com}


\credit{}

\affiliation[1]{organization={Department of Mathematics, University of Sfax},
            addressline={BP 1171}, 
            city={Sfax},
            postcode={3000}, 
            country={Tunisia}}

\author[2]{Hannes~Gernandt}[orcid=0000-0001-7364-4606]


\ead{gernandt@uni-wuppertal.de}


\credit{}

\affiliation[2]{organization={University of Wuppertal}, 
            addressline={Gau\ss stra\ss e 20}, 
            city={Wuppertal},
            postcode={42119}, 
            country={Germany}}
\author[1]
{Omaima~Kchaou}[orcid=0000-0001-7681-3591]
\ead{kchaouomaima53@gmail.com}
\credit{}

\author[3]{Carsten~Trunk}[orcid=0000-0003-3023-5567]

\cormark[1]
\ead{carsten.trunk@tu-ilmenau.de}
\credit{}

\affiliation[3]{organization={Institut für Mathematik, Technische Universität Ilmenau},
            addressline={Postfach 10 05 65}, 
            city={Ilmenau},
            postcode={98684}, 
            country={Germany}}
\author[3]{Mitsuru~Wilson}[orcid=0009-0003-7955-1218]
\ead{mitsuru14111@gmail.com}
\credit{}
\cortext[4]{Corresponding author}



\begin{abstract}
Linear relations are multivalued generalizations of linear operators whose spectral properties are closely linked to the stability of differential-algebraic equations. Consequently, reliable spectral enclosures are of considerable interest. In this note, we extend the established concept of the block numerical range from operators to linear relations by means of a range representation. Unlike the classical numerical range, the resulting block numerical range may be nonconvex and can therefore detect gaps in the spectrum. We establish fundamental properties of this new notion and derive corresponding spectral enclosure results. Finally, we illustrate the theory with a numerical example comparing the block numerical ranges obtained from different constructions. As an application, we obtain an explicit left-half-plane spectral bound for a class of port-Hamiltonian DAEs. 
\end{abstract}

\begin{highlights}
\item Block numerical ranges are extended from operators to linear relations.
\item Spectral enclosures are obtained via block numerical ranges in range representation.

\item Resolvent block ranges give unconditional spectral enclosures.

\item A port-Hamiltonian DAE application yields explicit stability half-plane bounds.
\end{highlights}
\maketitle
\section{Introduction}
Linear relations constitute a natural generalization of linear operators:
they contain graphs of single-valued operators as a special case, while also
allowing for multivalued behavior; see, for example,
\cite{Ammar2025,zbMATH06521173,zbMATH03447605,zbMATH03481315,
cross1998multivalued, zbMATH08092755, zbMATH03953636, von1951functional}.
Their systematic study goes back to Arens and Calkin \cite{arens1961operational, zbMATH02511052}. Since then, linear relations have
become an important tool in spectral and operator theory; we refer to
\cite{gernandt2020invariance} for a recent overview.

A particularly important feature of linear relations is their close
connection with operator pencils. If a linear relation in a Hilbert space $X$ admits a range
representation
\begin{align}
\label{linear relation: range representation}
    \Ll
    =
    \ran
    \begin{bmatrix}
        \mathcal B\\
        \mathcal A
    \end{bmatrix}
    =
    \bigl\{
        (\mathcal Bz,\mathcal Az):z\in X
    \bigr\},
\end{align}
for bounded linear operators $\mathcal A$ and $\mathcal B$ in $X$  cf. \cite{zbMATH07170295, zbMATH06521173, zbMATH08092755}, 
then its finite spectral points are closely related to the spectral values of
the pencil $\lambda\mathcal B-\mathcal A$, see, for example, \cite{gernandt2021spectrum}. 
This connection is especially
relevant for differential-algebraic equations, whose stability and transient
behavior are governed by the spectral and pseudospectral properties of
associated operator or matrix pencils. In this context, we mention the
pseudospectral approach of Embree and Keeler
\cite{EmbreeKeeler17}. As in the operator case, however, the spectrum of a
linear relation or operator pencil is generally difficult to determine
explicitly. Spectral enclosures are therefore of considerable practical and
theoretical importance.

A classical tool for obtaining such enclosures is the numerical range. For a
linear relation $\Ll$ in a Hilbert space $X$, its numerical range is defined
by
\begin{equation}
\label{Numerical range_Linear relation}
    W(\Ll)
    :=
    \bigl\{
        \langle y,x\rangle:
        (x,y)\in\Ll,\ \|x\|=1
    \bigr\}.
\end{equation}
If $\Ll$ is given by the range representation
\eqref{linear relation: range representation}, this can equivalently be written as
\begin{equation*}
    W(\Ll)
    =
    W(\mathcal B,\mathcal A)
    :=
    \bigl\{
        \langle \mathcal Az,\mathcal Bz\rangle:
        z\in X,\ \|\mathcal Bz\|=1
    \bigr\}.
\end{equation*}
The numerical range of a linear relation is convex; see
\cite{Amel73,article,rofe1990numerical}. Moreover, it provides the spectral
inclusions
\[
    \sigma_p(\Ll)\subseteq W(\Ll),
    \qquad
    \sigma_{\mathrm{ap}}(\Ll)
    \subseteq
    \overline{W(\Ll)};
\]
where $\sigma_p(\Ll)$ and  $\sigma_{\ap}(\Ll)$ denote the point spectrum and the approximate point spectrum of the linear relation $\Ll$, respectively. Further recent work on
numerical ranges and pseudospectral questions for linear relations can be
found in \cite{Ammar2025}.

Despite these useful properties, the classical numerical range may give a
rather coarse spectral enclosure. Its convexity prevents it, for instance,
from detecting spectral gaps or separating distinct spectral components.
This limitation has motivated the development of refined, generally
nonconvex numerical range concepts in the operator setting.

One such refinement is the \emph{quadratic numerical range}, introduced for
$2\times2$ block operator matrices in \cite{langer1998spectral} and further
studied in \cite{langer2001new,tretter2008spectral}. Unlike the classical
numerical range, the quadratic numerical range is generally nonconvex. It is
contained in the classical numerical range and, under the corresponding
spectral-enclosure hypotheses, provides substantially sharper information.
In particular, it may detect spectral gaps and can remain bounded in
situations in which the classical numerical range is unbounded. Numerical
methods for computing quadratic numerical ranges have also been investigated;
see, for example, \cite{Jacob2025}.

The quadratic numerical range was subsequently generalized to the
\emph{block numerical range}. Let
\[
    X=X_1\oplus\cdots\oplus X_n
\]
be an orthogonal decomposition of a Hilbert space and let $\mathcal A$ admit
the corresponding block operator representation
\begin{equation*}
    \mathcal A
    =
    \begin{bmatrix}
        A_{11} & \cdots & A_{1n}\\
        \vdots &        & \vdots\\
        A_{n1} & \cdots & A_{nn}
    \end{bmatrix}.
\end{equation*}
Let $\mathcal{S}_{X_i}$ denote the unit sphere in $X_i$, $i=1,\ldots,n.$ For $f=(f_1,\ldots,f_n)
    \in
    \mathcal S_{X_1}\times\cdots\times\mathcal S_{X_n}$, define the isometry
\[
    F_f:\mathbb C^n\longrightarrow X,
    \qquad
    F_fc=\sum_{j=1}^n c_jf_j,
\]
and the compressed matrix
$\mathcal A_f
    :=
    F_f^*\mathcal AF_f
    =
    \bigl(
        \langle A_{ij}f_j,f_i\rangle
    \bigr)_{i,j=1}^n$. The block numerical range \cite[Definition 2.1]{TretW03} is then
\begin{equation}
\label{eq:block_range_matrix}
    W^n(\mathcal A)
    :=
    \bigcup_{f\in
    \mathcal S_{X_1}\times\cdots\times\mathcal S_{X_n}}
    \sigma(\mathcal A_f),
\end{equation}
where $\sigma(\mathcal A_f)$ is the spectrum of $\mathcal A_f$, i.e., the set of all eigenvalues of the matrix $\mathcal A_f$. For $n=2$, this reduces to the quadratic numerical range. The block
numerical range was introduced for bounded block operator matrices in
\cite{tretter2008spectral, TretW03} and its main property is that it is a superset of the spectrum $\sigma(\A)$ and the point spectrum $\sigma_p(\A)$, i.e.\ 
\begin{equation}\label{5a}
\sigma_p(\mathcal A)\subseteq W^n(\mathcal A), \qquad \sigma(\mathcal A)\subseteq \overline{W^n(\mathcal A)}  
\end{equation}
but a subset of the numerical range of the operator $\A,$ 
\[
    W^n(\mathcal A)\subseteq W(\mathcal A),
\]
Subsequently, the block numerical range was extended to various classes of unbounded block operator matrices in 
\cite{rasulov2018spectral, tretter2009spectral}. In general, $W^n(\mathcal A)$ need not be convex. This makes block numerical ranges
capable of providing sharper spectral enclosures and of detecting spectral
gaps that are invisible to the ordinary numerical range. Related refined
numerical range and pseudospectral concepts for operator functions and their
discretizations have been considered in
\cite{Frommer2021, Gerhat2022}.

While the ordinary numerical range has been studied extensively for linear
relations, corresponding quadratic and block numerical range \cite{ rasulov2018spectral, tretter2008spectral, tretter2009spectral, TretW03} concepts have,
to the best of our knowledge, not yet been developed systematically. The
purpose of the present paper is to close this gap.

Our construction is based on a range representation
\eqref{linear relation: range representation}. For the same orthogonal decomposition
$X=X_1\oplus\cdots\oplus X_n$ and
$f\in\mathcal S_{X_1}\times\cdots\times\mathcal S_{X_n}$, we define
\[
    \mathcal A_f=F_f^*\mathcal AF_f,
    \qquad
    \mathcal B_f=F_f^*\mathcal BF_f,
\]
and associate with these matrices the truncated linear relation
\[
    \Ll_f
    :=
    \ran
    \begin{bmatrix}
        \mathcal B_f\\
        \mathcal A_f
    \end{bmatrix}
    \subseteq
    \mathbb C^n\times\mathbb C^n.
\]
The corresponding \emph{range-representation block numerical range} is
defined by
\begin{equation}
\label{eq:intro-block-range-relation}
    W^n(\mathcal B,\mathcal A)
    :=
    \bigcup_{\substack{
        f\in
        \mathcal S_{X_1}\times\cdots\times\mathcal S_{X_n}\\
        \rho(\Ll_f)\neq\emptyset
    }}
    \sigma(\Ll_f).
\end{equation}
Here $\sigma(\Ll_f)$ and $\rho(\Ll_f)$ stand for the spectrum and resolvent set, respectively of the linear relation $\Ll$. Note (cf. \cite{SandSW05}) that $\rho(\Ll_f)\neq \emptyset$ implies that $\sigma(\Ll_f)=\sigma_p(\Ll_f)$ which, in this case, consists of at most $n$ elements.
Furthermore, $\infty$ belongs to $W^n(\mathcal B,\mathcal A)$ whenever
$\mathcal B_f$ is singular and $\rho(\Ll_f)\neq\emptyset$ for at least one admissible block vector $f$.
The notation in \eqref{eq:intro-block-range-relation} emphasizes that this
construction may depend on the chosen range representation of $\Ll$.

The main contributions of the paper can be summarized as follows:
\begin{enumerate}
    \item
    We introduce the range-representation block numerical range
    $W^n(\mathcal B,\mathcal A)$ and establish its basic properties, and its
    connection with the block numerical range of the pencil
    $\lambda\mathcal B-\mathcal A$, and its monotonicity with respect to
    refinements of the underlying block decomposition.

    \item
    We prove spectral inclusions for linear relations. In particular, under the conditions $0\notin\overline{W^n(\mathcal B)}$ or $\mul\Ll=\mul\Ll^*$, the
    point spectrum satisfies
    \[
        \sigma_p(\Ll)
        \subseteq
        W^n(\mathcal B,\mathcal A),
    \]
    and,
    the full spectrum is enclosed by
\begin{align*}
    \sigma(\Ll)
        \subseteq
        \overline{W^n(\mathcal B,\mathcal A)}.\end{align*}
        \item We show
        \begin{align*}
        W^{1}(\mathcal{B},\A)\setminus\{\infty\} = w(\A,\mathcal{B}),
        \end{align*}
    where $w(\A,\mathcal{B})$ is the numerical range of the operator pencil $\lambda \mathcal{B}- \A$ introduced in \cite{zbMATH07330058}, see also \eqref{Four} below. And, if $\mul\Ll=\mul\Ll^*$, one has  
        \begin{align*}
         W^2(\mathcal{B}, \A)\backslash \{\infty\}=W(\Ll),
        \end{align*}
which indicates that $W^2(\mathcal{B},\A)$ is a natural generalization.
    \item
    For relations with nonempty resolvent set, we study particular range representations leading to  spectral enclosures.
\end{enumerate}

The paper is organized as follows. In Section~\ref{sec:preliminaries}, we
recall basic facts about linear relations and their spectra. In
Section~\ref{Section: Block}, we introduce the block numerical range based on
a range representation and establish its elementary properties. In Section~\ref{Spectral_inclusion_via range representation} we establish spectral inclusion results for the block numerical range via a range representation.
In Section~\ref{Arens_decomposition}, we discuss the special structure arising from the Arens decomposition.
In Section~\ref{sec:resolvent-block-numerical-range}, we study the block range based on a resolvent representation and prove spectral enclosures that hold
for every linear relation with nonempty resolvent set. In
Section~\ref{example_comparison}, we present finite-dimensional numerical
examples comparing the ordinary, quadratic, block, and resolvent block
numerical ranges. Finally, in Section~\ref{sec:application-ph}, we apply the block numerical range to a class of port-Hamiltonian differential-algebraic equations, obtaining an explicit left-half-plane spectral enclosure.

\section{Preliminaries on linear relations}
\label{sec:preliminaries}
\noindent

We begin by recalling some basic notions concerning linear relations. Let $X$ be a complex Hilbert space with inner product $\langle \cdot, \cdot\rangle$ and the norm $\|\cdot\|$ induced by the inner product. If $F\subset X$, by $F^\perp $ denote the orthogonal complement of $F$. By abusing the notation, the inner product on the Cartesian product $X^2:= X \times X$ is denoted also by $\langle\cdot, \cdot\rangle$,
\begin{align*}
	\langle (x,y), (x',y')\rangle = \langle x,x'\rangle +\langle y,y'\rangle, \ \  (x,y) , (x',y') \in X^2.
\end{align*}
The orthogonal projection onto a closed subspace $Y\subset X$ is denoted by $P_{Y}$. For closed subspaces $Y$ and $Z$ of the Hilbert space $X$, $Y\oplus Z$ denotes the orthogonal sum and $Y\ominus Z:=Y\cap Z^\perp$ is orthogonal subtraction. 

A linear relation in $X$ is a subspace $\Ll\subset X\times X$. It is
called closed if $\Ll$ is a closed subspace of $X \times X$. For a linear relation $\Ll$, we define the \emph{domain}, the \emph{range}, the \emph{kernel} and the \emph{multivalued part} of $\Ll$ as
\begin{eqnarray*}
	\dom\Ll&:= &\{x\in X ~ :~ (x,y) \in \Ll ~ \text{for some} ~ y\in X\};\\
	\ran\Ll&:= &\{y\in X ~ :~ (x,y) \in \Ll ~ \text{for some} ~ x\in X\};\\
	\ker\Ll&:= &\{x\in X ~ :~ (x,0) \in \Ll\};\\
	\mul\Ll&:= &\{y\in X ~ :~ (0,y) \in \Ll\}.
\end{eqnarray*}
A linear relation $\Ll$ is the graph of an operator if and only if $\mul\Ll = \{0\}$. Further, we define the \emph{inverse} $\Ll^{-1}$ and the \emph{adjoint} $\Ll^{*}$ as follows
\begin{eqnarray*}
    \Ll^{-1}&:=&\{(y,x) \ :\  (x,y) \in \Ll \}; \\
    \Ll^*&:=& \{(y,y')\in X\times X ~:~ \langle x',y\rangle=\langle x,y'\rangle \ \text{for all} \ (x,x')\in \Ll\}.
\end{eqnarray*}
A well-developed spectral theory for closed linear relations has been established; see \cite{zbMATH04015515, zbMATH03953636}. Its basic framework is analogous to that of operators: the expression $\Ll-\lambda I$ which is usually abbreviated as $\Ll-\lambda$, with
\begin{equation*}
	\Ll - \lambda := \{(x,y-\lambda x): (x,y)\in \Ll\}.
\end{equation*}
Together with $\Ll$, also $\Ll-\lambda$ and, thus, $(\Ll-\lambda)^{-1}$, are closed linear relations.  We call $\lambda \in \C$ a resolvent point if
$ \ker(\Ll-\lambda) =\{0\}$ and $ \ \ran(\Ll-\lambda)= X. $
Hence, $(\Ll-\lambda)^{-1}$ is a linear operator defined on $X$ and it is closed.
The Closed Graph Theorem implies that $(\Ll-\lambda)^{-1}$ is the graph of a bounded operator. Furthermore, we say that $\infty$ is a resolvent point of $\Ll$ if $\mul\Ll=\{0\}$ and $\dom\Ll=X$.

Moreover, if $\mul\Ll=\{0\}$ then the \emph{point spectrum} $\sigma_p(\Ll)$ is given by 
\begin{equation*}
\sigma_p(\Ll):=
	   \{\lambda\in\C~:~\ker(\Ll - \lambda)\neq\{0\}\}
\end{equation*}
or, equivalently $\lambda\in\sigma_p(\Ll)$ if $(x,\lambda x)\in\Ll$ for some $x\in X\setminus\{0\}$. If $\mul\Ll\neq\{0\}$, then we include $\infty$ in the point spectrum $\sigma_p(\Ll)$. For $\lambda \in \C,$ one has  
\begin{align*}
    \ker(\Ll-\lambda)= \ker(\Ll^{-1}-\lambda^{-1}), \quad \lambda \neq 0, \quad \hbox{and} \quad \ker \Ll=\mul \Ll^{-1}.
\end{align*} 
Therefore, it follows that 
\begin{align}
\lambda \in \sigma_p(\mathcal{L}) &\Longleftrightarrow \lambda^{-1} \in \sigma_p(\mathcal{L}^{-1}), \quad \lambda \neq 0, \label{point spectrum: inverse linear relation}\\
\intertext{and}
0 \in \sigma_p(\mathcal{L}) &\Longleftrightarrow \infty \in \sigma_p(\mathcal{L}^{-1}). \label{point spectrum: inverse linear relation_infinity} 
\end{align}
Let $\Ll$ and $\Rr$ be linear relations in $X$, we define the \emph{sum} $\Ll+ \Rr$, and the \emph{product} $\Rr \Ll$ by 
\begin{eqnarray*}
    \Ll+ \Rr &:=&\{(x,y+z)~:~(x,y) \in \Ll, \ (x,z) \in \Rr \}; \\
    \Rr \Ll &:=&\{(x,z)~:~(x,y) \in \Ll, \ (y,z) \in \Rr \}.
\end{eqnarray*}
We adopt the following particular representation of a closed linear relation whose resolvent set is non-empty \cite[Proposition~A.2.4]{Haas06}.

\begin{proposition}
\label{prop:relationfrompsres}
	Suppose that $\Ll$ is a closed linear relation in a Banach space $X$ whose resolvent is nonempty, $\rho(\Ll)\neq\emptyset$. Then,
    \begin{align*}
			\Ll  = \ran
			\begin{bmatrix}
				(\Ll-\mu)^{-1}\\ I+\mu(\Ll-\mu)^{-1}
			\end{bmatrix} \quad \text{and} \quad  \begin{matrix} \ker (\Ll-\mu)^{-1} &=& \mul \Ll \\ 
			\ran (\Ll-\mu)^{-1} &=& \dom \Ll,\end{matrix}
\end{align*}
for all $\mu\in\rho(\Ll)$.
\end{proposition}

The following result guarantees that linear relations in finite-dimensional space have finite spectrum \cite{SandSW05}. 
\begin{proposition}
\label{prop:finite_spectrum}
	Let $\Ll$ be a linear relation in $\C^n$ with $\rho(\Ll)\neq\emptyset$. Then the spectrum of $\Ll$ is purely point and finite. That is, $\sigma(\Ll) = \sigma_p(\Ll)$ and $\sigma(\Ll)$ is a finite set. Furthermore, if $\Ll=\ran\left[\begin{smallmatrix}
	    \mathcal{B}\\ \mathcal{A}
	\end{smallmatrix}\right]$ for some $\mathcal{A},\mathcal{B}\in\mathbb{C}^{n\times n}$, then the following is equivalent:
\begin{enumerate}[(a)]
\item $\lambda\in\sigma_p(\mathcal{L})\setminus\{\infty\}$;
\item $\det(\lambda\mathcal{B}-\mathcal{A})=0$;
\item $0\in\sigma(\lambda\mathcal{B}-\mathcal{A})$. 
\end{enumerate}
Furthermore, $\infty\in\sigma_p(\mathcal{L})$ if and only if $\det(\mathcal{B})=0$, or equivalently, $0\in\sigma(\mathcal{B})$.
\end{proposition}

The approximate point spectrum of a linear relation is defined as in  \cite{azizov2011spectral}.
\begin{definition}
	Let $\mathcal{L}$ be a linear relation in $X$. The \emph{approximate point spectrum} of $\mathcal{L}$ is given by 
	\begin{equation*}
		\sigma_{\ap}(\mathcal{L})=\{\lambda \in \C: \exists \ (x_n, \tilde{x}_n) \in\mathcal{L}, \ \ \textup{such that } \ \|x_n\|=1 \ \textup{and} \ \lim_{n\to\infty} \|\tilde{x}_n-\lambda x_n \|=0 \}
	\end{equation*} and $\infty\in\sigma_{\ap}(\mathcal{L}) $ if $0\in\sigma_{\ap}(\Ll^{-1}).$
\end{definition}

   \section{Block numerical range of linear relations}\label{Section: Block}
\noindent 
In this section, we introduce the notion of the block numerical range based on a range representation, starting with the following definition. 

\begin{definition}
\label{definition:truncated_LR}
	Let $\mathcal{A}:X\to X$ and $\mathcal{B}:X\to X$ be bounded linear operators in a Hilbert space $X$ with orthogonal decomposition $X=\bigoplus_{i=1}^nX_i$. For $f = (f_1, \ldots, f_n) \in \mathcal{S}_{X_1} \times \cdots \times \mathcal{S}_{X_n}$ with $f_i\in X_i$. Define the isometry
$$
F_f:\mathbb{C}^n\longrightarrow X,
\qquad
F_f c:=\sum_{j=1}^n c_jf_j,
$$ and $\mathcal{A}_{f}\in \mathcal{M}_n(\mathbb{C})$ by
	\begin{align*} \label{matrix A_f definition}
		\mathcal{A}_{f}:=
F_f^* \A F_f= \begin{bmatrix}
		\langle \A f_1,f_1\rangle  & \cdots& \langle \A f_n, f_1\rangle \\ \vdots &  & \vdots \\ \langle \A f_1,f_n \rangle  & \cdots& \langle \A  f_n , f_n \rangle
		\end{bmatrix} \in\mathbb{C}^{n\times n}.
	\end{align*}
	Let $\Ll$ be a closed linear relation in a Hilbert space $X$ given by \eqref{linear relation: range representation}. Then the \emph{truncated linear relation} is given by 
	\begin{equation}
    \label{eq:truncated_relation}
		\Ll_f:=\ran
		\begin{bmatrix}
			\mathcal{B}_f\\ \A_f
		\end{bmatrix}\subseteq\mathbb{C}^n\times\mathbb{C}^n.
	\end{equation}
\end{definition}
Note that in Defintion \ref{definition:truncated_LR} all subspaces $X_i$ are closed. This is due to the continuity of the inner product of the Hilbert space $X.$
\begin{definition} \label{Definition: Block numerical range} Let $\mathcal{L}$ be a closed linear relation in $X$ given by  \eqref{linear relation: range representation}. Then the block numerical range of~$\Ll$ (w.r.t.\ the range representation \eqref{linear relation: range representation}) is defined as
\begin{align*}
W^n(\mathcal{B}, \A):=\bigcup\limits_{\substack{f\in\prod_{i=1}^n\mathcal{S}_{X_i}, \\ \rho(\mathcal{L}_f)\neq\emptyset}} \sigma(\mathcal{L}_f).
\end{align*}
Moreover, $\infty\in W^n(\mathcal{B}, \A)$ if there exists $f\in\prod_{i=1}^n\mathcal{S}_{X_i}$ such that  $\mathcal{B}_f$ is singular and $\rho(\mathcal{L}_f)\neq\emptyset.$
\end{definition}
 The following example shows that even for linear relations that are graphs of matrices, the truncated relation $\Ll_f$ might have the entire complex plane as spectrum, which shows why the restriction to those $f\in X$ satisfying $\rho(\Ll_f)\neq \emptyset$ is necessary.
\begin{example}
\label{ex:singular}
Consider
\begin{align*}
\mathcal{A}=\begin{bmatrix}
    0&0&1\\ 1&0&0 \\ 0&1&0 
\end{bmatrix},\quad \mathcal{B} = \begin{bmatrix}
      0&0&1\\ 0&1&0 \\ 1&0&0 
\end{bmatrix}. 
\end{align*}
Then the relation $\Ll=\ran\begin{bmatrix}
    \mathcal{B}\\ \mathcal{A}
\end{bmatrix}=\ran\begin{bmatrix}
    I\\ \mathcal{A}\mathcal{B}^{-1}
\end{bmatrix}$ in $\R^3$ is the graph of the matrix $\mathcal{A}\mathcal{B}^{-1}$ and its point spectrum consists of finitely many points.
Using the canonical unit vectors $e_1,e_2$, and $e_3$ in $\R^3$ we split $X_1=\Span\{e_1\}$ and $X_2=\Span\{e_2,e_3\}$ and we choose $f_1=e_1$ and $f_2=e_2$ then 
\[
\Ll_f=\ran\begin{bmatrix}
    \mathcal{B}_f\\ \mathcal{A}_f 
\end{bmatrix} =\ran\begin{bmatrix}
    \langle \mathcal{B}e_1,e_1 \rangle & \langle \mathcal{B}e_2,e_1 \rangle\\\langle \mathcal{B}e_1,e_2 \rangle& \langle \mathcal{B}e_2,e_2 \rangle \\  \langle \mathcal{A}e_1,e_1 \rangle & \langle \mathcal{A}e_2,e_1 \rangle\\\langle \mathcal{A}e_1,e_2 \rangle& \langle \mathcal{A}e_2,e_2 \rangle
\end{bmatrix}=\ran\begin{bmatrix}
    0 & 0 \\ 0& 1 \\  0 & 0 \\ 1& 0
\end{bmatrix}.
\]
It follows that $$\mathcal{B}_f\begin{bmatrix}
    \lambda \\ 1
\end{bmatrix}=\begin{bmatrix}
    0 \\ 1
\end{bmatrix} \  \text{and} \ \mathcal{A}_f\begin{bmatrix}
    \lambda \\ 1
\end{bmatrix}=\lambda\begin{bmatrix}
    0 \\ 1
\end{bmatrix}.$$
As a consequence, we have $(e_2,\lambda e_2)\in\Ll_f$  and hence $(e_2,0)\in\Ll_f-\lambda$ for all $\lambda\in\mathbb{C}$. Thereby $\sigma(\Ll_f)=\mathbb{C}$. 
\end{example}
We define $W^n(\mathcal{B}^*, \A^*)^*:=\{ \overline{\lambda}\in \C : \lambda\in W^n(\mathcal{B}^*, \A^*) \}$ for the complex conjugate of $W^n(\mathcal{B}^*, \A^*)$.
\begin{proposition}
    \label{Remark: Range_Kernel} Let $\mathcal{A}:X\to X$ and $\mathcal{B}:X\to X$ be bounded linear operators in a Hilbert space $X$ with adjoints $\A^*$ and $\mathcal{B}^*$, respectively. We consider the linear relation $$ \Ll'= \ker[\A, -\mathcal{B}],
$$ called the \emph{kernel representation}. Then the set $W^n(\mathcal{B}, \A)$ has the equivalent representation\begin{align}\label{Def_Equivalence_Kernel}
W^n(\mathcal{B}, \A)=\bigcup\limits_{\substack{f\in\prod_{i=1}^n\mathcal{S}_{X_i}, \\ \rho(\mathcal{L}'_f)\neq\emptyset}} \sigma(\mathcal{L}'_f), \quad \text{with} \ \mathcal{L}'_f:= \ker[\A_f, -\mathcal{B}_f].
\end{align} Moreover,  \begin{align}\label{Adjoint} W^n(\mathcal{B}, \A)=  W^n(\mathcal{B}^*, \A^*)^*.\end{align}
\end{proposition}
\begin{pf}
 Let $\lambda \in W^n(\mathcal{B}, \A)$, then from Definition \ref{Definition: Block numerical range} there exists $f=(f_i)_{i=1}^n\in\prod_{i=1}^n\mathcal{S}_{X_i}$ such that $\lambda\in \sigma(\mathcal{L}_f)$ and $\rho(\mathcal{L}_f)\neq\emptyset$, which is equivalent, by  \cite[Proposition 3]{gernandt2021spectrum} to $\lambda\in \sigma(\mathcal{L}'_f)$ with $\rho(\mathcal{L}'_f)\neq\emptyset$. This gives \eqref{Def_Equivalence_Kernel}. Obviously, one has  
\begin{align*}
\ker[\A, -\mathcal{B}]= \bigg( \ran\begin{bmatrix}\mathcal{B}^*\\ \A^*\end{bmatrix}\bigg)^* 
	\end{align*}
 and hence one has \eqref{Adjoint}.
\end{pf}
In the following we collect some elementary properties of the block numerical range of a linear relation. 
\begin{proposition}
        Let $\mathcal{L}$ be a closed linear relation in $X$ given by  \eqref{linear relation: range representation}. Then the following hold.
\begin{enumerate}[(a)]
    \item $W^n( \mathcal{B}, \alpha \A+ \beta \mathcal{B})= \alpha W^n(\mathcal{B}, \A)+\beta, \quad$ for $\alpha, \beta \in \mathbb{C}$, $\alpha\neq 0$.
    \item We have \begin{align*}
        W^n(\A, \mathcal{B})\backslash \{0, \infty\}= \{\lambda ^{-1} \in \C \ |\ \lambda \in  W^n(\mathcal{B}, \A), \ \lambda \neq 0\}. \end{align*}
    Moreover,  \begin{align}\label{infinity_inverse}
         \infty \in W^n(\A, \mathcal{B})\iff 0\in W^n(\mathcal{B}, \A). \end{align}
    \item  We have \begin{eqnarray*}
    W^n(\mathcal{B}, \A)& = & \bigl\{
\lambda \in \mathbb{C} \cup \{\infty\} :
\begin{aligned}[t]
& \ \text{there exists an} \ f=(f_i)_{i=1}^n \in\prod_{i=1}^n\mathcal{S}_{X_i}, \ i=1,\ldots,n,\\
& \ \ \text{such that} \ \rho(\Ll_f)\neq\emptyset \ \text{and}\ \det( \A_f -\lambda \mathcal{B}_f)=0, \text{ if}\ \lambda \in \C,  \\
& \hskip 4.25cm \text{ or } \det(\mathcal{B}_f)=0, \text{ if} \ \lambda=\infty \bigr\}.
\end{aligned} 
\end{eqnarray*}
       
\end{enumerate}
\end{proposition} 
\begin{pf}
\begin{enumerate}[(a)]
    \item  For $f=(f_i)_{i=1}^n\in\prod_{i=1}^n\mathcal{S}_{X_i}$ we have \begin{eqnarray*}
         \alpha \sigma(\mathcal{L}_f)+\beta &=& \sigma(\alpha \mathcal{L}_f+\beta) \nonumber  = \sigma\bigg( \ran
    \begin{bmatrix}
         \mathcal{B}_f\\[1ex]
    \alpha     \mathcal A_f
    \end{bmatrix}+ \beta I\bigg)  \\
     &=& \sigma(\{(  B_f x, \alpha A_f x+ \beta B_f x) : x\in X\})  = \sigma\bigg( \ran
    \begin{bmatrix}
           \mathcal B_f\\
        \alpha \mathcal  A_f+\beta   B_f
    \end{bmatrix}\bigg).
     \end{eqnarray*} 
     Now, (a) follows by taking the union in the above over all $f=(f_i)_{i=1}^n\in\prod_{i=1}^n\mathcal{S}_{X_i}$ and the fact that $\rho(\Ll_f)\neq \emptyset$ if and only if $\rho(\alpha\Ll_f+\beta)\neq \emptyset.$
    \item  Let $\lambda \in W^n(\mathcal{B}, \A)\backslash \{\infty\}$, $\lambda \neq 0$, there exists $f=(f_i)_{i=1}^n\in\prod_{i=1}^n\mathcal{S}_{X_i}$ such that $\lambda\in \sigma(\mathcal{L}_f)$ and $\rho(\Ll_f)\neq\emptyset$, $\lambda \neq 0$, which is equivalent, by \eqref{point spectrum: inverse linear relation}, to $\lambda^{-1}\in \sigma(\Ll_f^{-1})$ since $\sigma(\mathcal{L}_f)=\sigma_p(\mathcal{L}_f)$ by Proposition \ref{prop:finite_spectrum}. Taking the union of $\sigma(\Ll_f^{-1})$ over all $f=(f_i)_{i=1}^n\in\prod_{i=1}^n\mathcal{S}_{X_i}$ and observing that $\rho(\Ll_f^{-1})\neq \emptyset,$ we obtain $\lambda^{-1} \in W^n(\A, \mathcal{B}).$ The opposite inclusion follows by the same argument. The equivalence in \eqref{infinity_inverse} follows directly from \eqref{point spectrum: inverse linear relation_infinity}.
    \item For $f\in\prod_{i=1}^n\mathcal{S}_{X_i}$, the claim follows immediately by Definition \ref{Definition: Block numerical range} and Proposition~\ref{prop:finite_spectrum}.
 \end{enumerate}
  \end{pf}
The following proposition establishes a connection between the block numerical range of the linear relation $\Ll$ as defined in \eqref{linear relation: range representation} and that of the corresponding operator pencil $\lambda \mathcal{B}-\A$.
\begin{proposition}\label{Perfekt}
    Let $\mathcal{L}$ be a closed linear relation in $X$ given by  \eqref{linear relation: range representation}. Then
    \begin{equation}\label{Linear relation/Pencil}
     W^n(\mathcal{B}, \A)\backslash \{\infty\}\subseteq\{\lambda \in \C : 0 \in  W^n(\lambda \mathcal{B}-\A)\}.\end{equation}
     If, in addition, $0\notin \overline{W^n(\mathcal{B})}$, then 
     \begin{equation}\label{Linear relation/Pencil*}
     W^n(\mathcal{B}, \A)\backslash \{\infty\}=\{\lambda \in \C : 0 \in  W^n(\lambda \mathcal{B}-\A)\}.
     \end{equation}
     \end{proposition}
    \begin{pf}
          Let $\lambda\in W^n(\mathcal{B}, \A)\backslash \{\infty\}$. Then there 
          exists $f\in\prod_{i=1}^n\mathcal{S}_{X_i}$ such that $\lambda \in \sigma(\Ll_f)$ and $\rho(\Ll_f)\neq\emptyset$. Thus, by Proposition~\ref{prop:finite_spectrum}, we have  $0\in \sigma(\lambda \mathcal{B}_f-\A_f)$. In view of the definition of the block numerical range in \eqref{eq:block_range_matrix} (cf.\ \cite[Definition 2.1]{TretW03}), it follows that  $0 \in  W^n(\lambda \mathcal{B}-\A)$. This shows \eqref{Linear relation/Pencil}.

          We show the opposite inclusion in
          \eqref{Linear relation/Pencil*}. Let $\lambda$ be such that $
              0 \in  W^n(\lambda \mathcal{B}-\A)$. By \eqref{eq:block_range_matrix}, it follows that there exists $f\in\prod_{i=1}^n\mathcal{S}_{X_i}$  such that 
              \begin{align}\label{Spectrum_Pencil}
              0\in \sigma(\lambda \mathcal{B}_f-\A_f). 
              \end{align}  
              By assumption, we have $0\notin \overline{W^n(\mathcal{B})}$. It follows from \cite[Lemma 6.3]{radl2014block} that $0\notin \overline{W^{n}(\mu \mathcal{B}-\A)}$  for large values of $\lvert\mu\rvert$. Thus, by \eqref{5a}, $0\notin \sigma(\mu \mathcal{B}_f-\mathcal{A}_f)$ for all $f\in\prod_{i=1}^n\mathcal{S}_{X_i}$ and hence the pencil $\lambda \mathcal{B}_f-\A_f$ is regular, which is equivalent to $\rho(\Ll_f)\neq \emptyset.$  By Proposition \ref{prop:finite_spectrum} and \eqref{Spectrum_Pencil}, we obtain $\lambda \in \sigma(\Ll_f)\setminus \{\infty\}.$ Taking the union over all $f\in\prod_{i=1}^n\mathcal{S}_{X_i}$such that $\rho(\Ll_f)\neq \emptyset$, we obtain $\lambda \in W^n(\mathcal{B}, \A)\backslash \{\infty\}.$ 
         \end{pf}
In what follows, we establish the relation between the block numerical range $W^{1}(\mathcal{B}, \A)$ and the numerical range for the operator pencil $\lambda \mathcal{B}- \A$ introduced in \cite{zbMATH07330058}. We recall the latter definition, which is given by
\begin{equation}\label{Four}
  w(\A,\mathcal{B}):=\left\{\frac{\langle \A x,x\rangle}{\langle \mathcal{B} x,x\rangle}\ :\ x\in X,\ \langle \mathcal{B}x,x\rangle\neq0\right\}.
\end{equation}
\begin{proposition}
\label{prop:scalar-block-range}
Let $\Ll$ be a closed linear relation in $X$ given by  \eqref{linear relation: range representation}. Then 
$$ 
W^{1}(\mathcal{B},\A)\setminus\{\infty\} = w(\A,\mathcal{B}). 
$$ 
\end{proposition}
\begin{pf}
 Let $\lambda \in W^{1}(\mathcal{B},\A)\setminus \{\infty \} $. Then there exists $f\in\mathcal{S}_{X}$ and $\rho( \Ll_f )\neq \emptyset$ such that 
 $$
 \lambda\in\sigma\bigg(\ran
    \begin{bmatrix}
         \langle \mathcal{B}f,f\rangle\\
        \mathcal  \langle \A f,f\rangle
    \end{bmatrix}\bigg).
    $$
    Then there  exists  $z\neq0$ with $(z,\lambda z) \in \Ll_f.$
    Hence,  $z=\langle\mathcal{B}f,f\rangle x,$ $\lambda z=\langle\A f,f\rangle x$, for some 
    $x\in \C$. This gives $\lambda=\frac{\langle\A f,f\rangle}{\langle\mathcal{B} f,f\rangle}$ and hence $\lambda\in w(\mathcal{A}, \mathcal{B})$.

    Let $\lambda \in w(\A,\mathcal{B})$. Then there exist $f\in X$ and $\langle \mathcal{B}f,f\rangle\neq0$ such that $\lambda=\frac{\langle\A f,f\rangle}{\langle\mathcal{B} f,f\rangle}.$ Consider the linear relation $
        \Ll_f= \{(\mathcal{B}_f z, \A_f z) : z\in \C\} $
        with $\mathcal{B}_f=\langle \mathcal{B}f,f\rangle$  and $\mathcal{A}_f=\langle \mathcal{A}f,f\rangle $ and hence  
        \begin{align*}
        \Ll_f= \{(\mathcal{B}_f z, \lambda \mathcal{B}_f z) : z\in \C\},
        \end{align*}  
        which is the graph of the scalar operator $\frac{\langle\A f,f\rangle}{\langle\mathcal{B}f,f\rangle}.$ Consequently, $\Ll_f-\mu$ is the graph of $\frac{\langle\A f,f\rangle}{\langle\mathcal{B}f,f\rangle}-\mu$ which is boundedly invertible if and only if $\frac{\langle\A f,f\rangle}{\langle\mathcal{B}f,f\rangle}\neq\mu$ and hence $\rho(\Ll_f)\neq \emptyset.$ 
        Therefore, 
        \begin{align*}
        \sigma(\Ll_f)=\bigg \{ \frac{\langle\A f,f\rangle}{\langle\mathcal{B}f,f\rangle}\bigg\}.
        \end{align*} 
        Consequently, $\lambda \in W^{1}(\mathcal{B},\A)\setminus \{\infty \}$ follows.
\end{pf}
Note that, in general, $W^1(\mathcal{B}, \A)$ does not recover the numerical range $W(\mathcal{B}, \A)$ for the operators $\mathcal{B}$ and $\A$ from a range representation \eqref{linear relation: range representation} of $\mathcal{L}$. This is illustrated by the following example.
\begin{example}
\label{ex:W-versus-W1}
Consider
\[
 \mathcal B=\begin{bmatrix}1&0\\0&4\end{bmatrix},\qquad
 \mathcal A=\begin{bmatrix}0&1\\0&0\end{bmatrix},\qquad
 \Ll=\ran\begin{bmatrix}\mathcal B\\\mathcal A\end{bmatrix}.
\]
Since $\mathcal B$ is positive definite,
$\Ll=\operatorname{gr}(\mathcal A\mathcal B^{-1})$, where
\[
 \mathcal A\mathcal B^{-1}=\begin{bmatrix}0&\tfrac14\\0&0\end{bmatrix}.
\]
For a unit vector $x=(x_1,x_2)^\top$,
\[
 \langle\mathcal A\mathcal B^{-1}x,x\rangle
 =\frac14x_2\overline{x_1}.
\]
Since $|x_1x_2|\le\tfrac12$, and every modulus between $0$ and
$\tfrac12$ can be obtained with an arbitrary phase,
\begin{equation}
\label{eq:example-W-small-disk}
 W(\mathcal B,\mathcal A)=W(\Ll)
 =\{z\in\C:|z|\le\tfrac18\}.
\end{equation}
On the other hand, Proposition~\ref{prop:scalar-block-range} gives
\[
 W^1(\mathcal B,\mathcal A)\setminus\{\infty\}
 =\left\{\frac{z_2{z_1}^{-1}}
 {1+4r^2}:(z_1,z_2)\neq0\right\}.
\]
With
$r=|z_2|/|z_1|$ and
\[
 \frac{r}{1+4r^2}\le\frac14,
\]
with equality at $r=\tfrac12$. The continuous function on the left assumes
every value in $[0,\tfrac14]$. Therefore
\begin{equation}
\label{eq:example-W1-large-disk}
 W^1(\mathcal B,\mathcal A)\setminus\{\infty\}
 =\{z\in\C:|z|\le\tfrac14\}.
\end{equation}
Moreover, $\infty\notin W^1(\mathcal B,\mathcal A)$ because
$\langle\mathcal Bf,f\rangle>0$ for every nonzero $f$. Thus
\eqref{eq:example-W-small-disk} and \eqref{eq:example-W1-large-disk} give a strict inclusion.
\end{example}

\section{Spectral inclusion}\label{Spectral_inclusion_via range representation}

In this section, we establish  spectral inclusion results for the block numerical range of the linear relation $\Ll$ in \eqref{linear relation: range representation}. 
The following result was established in \cite{radl2014block} for analytic operator functions. For the convenience of the reader, we include a proof adapted to our setting. 
\begin{lemma}\label{Equivalence: linear relation_Pencil}
    Let $\Ll$ be a closed linear relation in $X$ given by  \eqref{linear relation: range representation} and $\lambda \in \C$. Then \begin{align}\label{First implication}
        \lambda \in \overline{W^n(\mathcal{B}, \A)} \Longrightarrow 0 \in  \overline{W^n(\lambda \mathcal{B}-\A)}.
    \end{align}
    
    If, in addition, $0\notin \overline{W^n(\mathcal{B})}$, then
    \begin{align}\label{Equivalence}
        \lambda \in \overline{W^n(\mathcal{B}, \A)} \iff  0 \in  \overline{W^n(\lambda \mathcal{B}-\A)}.
    \end{align}
    \end{lemma}
\begin{pf}  
 Assume that $\lambda \in \overline{W^n(\mathcal{B}, \A)}.$ Then there exists a sequence $(\lambda_k)\subset W^n(\mathcal{B}, \A)$ such that $\lambda_k \rightarrow \lambda, k\rightarrow \infty$. By \eqref{Linear relation/Pencil}, we obtain that $$0\in W^n(\lambda_k \mathcal{B}-\A), $$ and hence there exists $f^k\in \prod_{i=1}^n\mathcal{S}_{X_i} $ such that $\det((\lambda_k \mathcal{B}-\A)_{f^k})=0$. Thus, $(\lambda_k \mathcal{B}-\A)_{f^k}$ is not invertible, then there exists a nonzero vector $h$ such that $(\lambda_k \mathcal{B}-\A)_{f^k}h=0,$ with $\|h\|=1.$ Therefore, 
 the minimum modulus of $(\lambda_k \mathcal{B}-\A)_{f^k}$ (cf.\ \cite{zbMATH03239038}) satisfies
 $$
 \gamma((\lambda_k \mathcal{B}-\A)_{f^k})= \inf_{\|x\|=1}\|(\lambda_k \mathcal{B}-\A)_{f^k}x\|=0.
 $$
  From Definition \ref{definition:truncated_LR}, we have $\mathcal{B}_f= F_f^* \mathcal{B} F_f$ and since $F_f$ is an isometry, it follows that \begin{align}\label{Bounded_Norm}
    \|\mathcal{B}_f\| \le \|\mathcal{B}\|, \quad \text{for} \
    f\in \prod_{i=1}^n\mathcal{S}_{X_i}\end{align} and therefore, \begin{eqnarray*}
      \|(\lambda_k \mathcal{B}_{f^k}-\A_{f^k})-(\lambda \mathcal{B}_{f^k}-\A_{f^k})\|&=& \|(\lambda_k-\lambda) \mathcal{B}_{f^k}\|\le |\lambda_k-\lambda| \|\mathcal{B}\|\rightarrow 0, \quad k\rightarrow \infty.
      \end{eqnarray*}
This implies that
\begin{eqnarray*}
\gamma((\lambda \mathcal{B}-\A)_{f^k})&=&| 
\gamma((\lambda \mathcal{B}-\A)_{f^k})-
\gamma((\lambda_k \mathcal{B}-\A)_{f^k})| \le \|
(\lambda_k \mathcal{B}-\A)_{f^k}-(\lambda \mathcal{B}-\A)_{f^k})
\|\rightarrow 0, \quad k\rightarrow \infty. 
\end{eqnarray*} 
Using the singular-value decomposition $(\lambda 
\mathcal{B}-\A)_{f^k}$ there exist unitary 
$U,V\in\mathbb{C}^{n\times n}$ such that  
$U^*(\lambda \mathcal{B}-\A)_{f^k}V=\operatorname{diag}(\sigma_1^k,\sigma_2^k,\ldots,\sigma_n^k)$ for some 
\[
\sigma_1^k=\|(\lambda \mathcal{B}-\A)_{f^k}\|\ge \ldots\ge \sigma_n^k=\gamma((\lambda \mathcal{B}-\A)_{f^k})\ge  0.
\] 
We estimate
\begin{align*}
\left|\det((\lambda \mathcal{B}-\A)_{f^k})\right|&
=\left|\det(U^*(\lambda \mathcal{B}-\A)_{f^k}V)\right|
=\sigma_1^k\sigma_2^k\cdot\ldots\cdot \sigma_n^k\le
\gamma((\lambda \mathcal{B}-\A)_{f^k})
\|(\lambda \mathcal{B}-\A)_{f^k}\|^{n-1}.
\end{align*}
By the same argument as in \eqref{Bounded_Norm}, it follows that $\|(\lambda \mathcal{B}-\A)_{f^k}\|\le \|(\lambda \mathcal{B}-\A)\|$. Consequently,  
$$
\left|\det((\lambda \mathcal{B}-\A)_{f^k})\right| \rightarrow 0, \quad k\rightarrow \infty.
$$ 
 Hence, by \cite[Lemma 4.3]{radl2014block}, this implies that $0 \in  \overline{W^n(\lambda \mathcal{B}-\A)}.$ This proves \eqref{First implication}. 
 
 In order to show  \eqref{Equivalence}, assume that $0\notin \overline{W^n(\mathcal{B})}$. Let $\lambda \in \C$ be such that 0 $\in  \overline{W^n(\lambda \mathcal{B}-\A)}$. It follows that $\lambda \in \overline{W^n(\mathcal{B}, \A)}$ by applying similar
arguments as in the corresponding implication of the proof of \cite[Lemma 4.5]{radl2014block}, after replacing the operator function $\Ll(\lambda)$ with $\lambda \mathcal{B}-\A$ and \eqref{Linear relation/Pencil*}. However, the condition $0\notin \overline{W^n(\mathcal{B})}$ is sufficient for the required assumption in terms of the leading coefficient.   \end{pf}

\begin{proposition}
\label{remark3}
    Let $\Ll$ be a closed linear relation in $X$ given by  \eqref{linear relation: range representation}. Then the following hold.
    \begin{enumerate}[(a)]
    \item
    If $0\notin \overline{W^n(\mathcal{B})}$, then $W^{n}(\mathcal{B},\A)$ is bounded.
    \item If $\infty\in\overline{W^n(\mathcal{B}, \A)}$ then $0 \in  \overline{W^n(\mathcal{B})}$.
    \end{enumerate}
    \end{proposition}
    
    \begin{pf}
        \begin{enumerate}[(a)]
    \item Since $0\notin \overline{W^n(\mathcal{B})}$, it follows from \cite[Lemma 6.3]{radl2014block} that $0\notin \overline{W^{n}(\lambda \mathcal{B}-\A)}$ for large values of $\lvert\lambda\rvert$. Hence, by \eqref{First implication}, $\lambda\notin \overline{W^{n}(\mathcal{B},\A)}$ for large values of $\lvert\lambda\rvert$. This implies that $W^{n}(\mathcal{B},\A)$ is bounded.
    \item Assume $\infty\in\overline{W^n(\mathcal{B}, \A)}$ and suppose that $0\notin\overline{W^n(\mathcal{B})}$. Hence, by (a), we obtain that  $W^n(\mathcal{B}, \A)$ is bounded, but then $\infty\notin\overline{W^n(\mathcal{B}, \A)}$ which is the desired contradiction. Therefore, $0\in \overline{W^n(\mathcal{B})}.$ 
    \end{enumerate}
        \end{pf}

We define $\sigma_p(\Ll^*)^*:=\{\lambda\in \C : \overline{\lambda}\in \sigma_p(\Ll^*) \}$ for the complex conjugate of $\sigma_p(\Ll^*)$.
\begin{proposition}\label{spectrum LR_implication}
Let $\Ll$ be a closed linear relation in $X$ given by  \eqref{linear relation: range representation}. Then
     \begin{align*}
   \sigma(\Ll)=\sigma_{\ap}(\Ll)\cup \sigma_p(\mathcal{L}^*)^*. 
\end{align*}
\end{proposition}
\begin{pf}
 Let $\lambda \in \C\backslash (\sigma_{\ap}(\Ll)\cup \sigma_p(\mathcal{L}^*)^*)$. Then, $\lambda \notin \sigma_{\ap}(\Ll)$ and $\overline{\lambda}\notin \sigma_p(\Ll^*)$. We have $\lambda \notin \sigma_{\ap}(\Ll)$, which gives that $\lambda \mathcal{B}-\A$ is bounded below, hence its range  $\ran (\lambda \mathcal{B}-\A)$ is closed and $\lambda\notin \sigma_p(\Ll),$ this implies that $\lambda \mathcal{B}-\A$ is injective. Since $\overline{\lambda}\notin \sigma_p(\Ll^*)$, it follows that $(\lambda \mathcal{B}-\A)^*=\overline{\lambda} \mathcal{B}^*-\A^*$ is injective and therefore $$ \overline{\ran(\lambda \mathcal{B}-\A)}=\ker((\lambda \mathcal{B}-\A)^*)^{\perp}= X.$$
Altogether we obtain  $ \ran(\lambda \mathcal{B}-\A)=\overline{\ran(\lambda \mathcal{B}-\A)}=\ker((\lambda \mathcal{B}-\A)^*)^{\perp}= X.$ Therefore, $\lambda \in \rho(\Ll),$ which implies $\lambda \in \C\backslash \sigma(\Ll).$ This proves the inclusion. 

 Since $\sigma_{\ap}(\Ll) \subseteq \sigma (\Ll)$, it remains to prove that $\sigma_p(\mathcal{L}^*)^*\subseteq \sigma(\mathcal{L})$. Let $\lambda \in \sigma_p(\mathcal{L}^*)^*$, then $\overline{\lambda} \in \sigma_p(\mathcal{L}^*)$. Assume, for contradiction, that  $\lambda \notin\sigma(\mathcal{L})$. Thus, $\lambda \in\rho(\mathcal{L})$ which by \cite[Proposition 1.3.10]{zbMATH07170295} is equivalent to $\overline{\lambda} \in\rho(\mathcal{L}^*)$. This implies that, $\overline{\lambda} \notin \sigma_p(\mathcal{L}^*)$ which is a contradiction. Therefore, $\lambda \in \sigma(\Ll).$ 
\end{pf}

Our first main theorem is the following.    \begin{theorem}\label{spectral_enclosures_pencils}
	Let $\Ll$ be a closed linear relation in $X$ given by  \eqref{linear relation: range representation}. If $0\notin \overline{W^n(\mathcal{B})}$  then the following statements hold.
     \begin{enumerate}[(a)]
    \item
     $\sigma_{p}(\Ll)\subseteq W^{n}(\mathcal{B}, \A)$. 
    \item
     $\sigma(\Ll) \subseteq \overline{W^{n}(\mathcal{B}, \A)}.$
    \item
        If $X = \widetilde{X}_1 \times \cdots \times \widetilde{X}_{m}$ is a refinement of $X = X_1 \times \cdots \times X_n$, then $$W^{m}(\mathcal{B}, \A)\subseteq W^n(\mathcal{B}, \A), \  m\geq  n\geq 1.$$
	\end{enumerate}
\end{theorem}
\begin{pf}
     \begin{enumerate}[(a)]
    \item We show that $\infty\notin\sigma(\mathcal{L})$. Suppose otherwise, then either $\mul\mathcal{L}\neq\{0\}$ or $\dom\Ll\neq X$ which yields $0\in\sigma(\mathcal{B})$.
Then, $0 \in \overline{W^n(\mathcal{B})}$ which is a contradiction. 
Therefore, $\infty\notin\sigma_p(\mathcal{L})$. 
        Let $\lambda \in \sigma_p(\mathcal{L})$. Then for some vector
        $x\in X$, $x \neq 0$, we have $(x, \lambda x)\in \Ll$. This
        implies that $x= \mathcal{B} y$ and $\lambda x = \mathcal{A} y$
        for some $y\in X$, $y \neq 0$ and hence 
        $(\lambda \mathcal{B}-\mathcal{A})y=0.$ Therefore, 
        $0\in \sigma_p(\lambda \mathcal{B}-\A)$. 
According to \cite[Theorem 6.1]{radl2014block}, it follows that $0\in W^n(\lambda \mathcal{B}-\A). $ This implies, by \eqref{Linear relation/Pencil*}, that  $\lambda\in W^n(\mathcal{B}, \A)$. 
    \item If $\lambda \in \sigma(\Ll)$, then it follows from Proposition \ref{spectrum LR_implication} that $\lambda\in \sigma_p(\Ll^*)^*$ or $\lambda \in \sigma_{\ap}(\Ll)$. If $\lambda\in \sigma_p(\Ll^*)^*$, then $\overline{\lambda}\in \sigma_p(\Ll^*).$ By (a), we obtain  $\overline{\lambda}\in W^n(\mathcal{B}^*,\A^*)$. Thus, \eqref{Adjoint} shows that  $\lambda\in W^n(\mathcal{B},\A).$ If $\lambda \in \sigma_{\ap}(\Ll)$, it follows that there exists a sequence $(x_k,y_k)\in \Ll$ with $\|x_k\|=1$, $k\in\mathbb{N}$ and $\lambda x_k- y_k\rightarrow 0,$ as $k\rightarrow \infty$. Together with \eqref{linear relation: range representation}, this gives $x_k=\mathcal{B}z_k$ and $y_k=\mathcal{A}z_k$ for some $z_k\in X$ and hence $\lambda \mathcal{B}z_k- \A z_k \rightarrow 0,$ as $k\rightarrow \infty$. Therefore, $0\in \sigma_{\ap}(\lambda \mathcal{B}- \A)$, together with \cite[Theorem 6.1]{radl2014block}, we obtain $0\in \overline{W^n(\lambda\mathcal{B}-\A)}$. Since $0\notin \overline{W^n(\mathcal{B})}$, it follows from \eqref{Equivalence} that 
    $\lambda \in \overline{W^n(\mathcal{B}, \A)}$.
    \item Let $\lambda \in  W^{m}(\mathcal{B},\A).$ Then, by \eqref{Linear relation/Pencil}, it follows that  $0 \in  W^{m}(\lambda \mathcal{B}-\A).$ According to \cite[Theorem 3.5]{TretW03}, we obtain $ W^{m}(\lambda \mathcal{B}-\A)\subseteq W^{n}(\lambda \mathcal{B}-\A).$ Hence, by \eqref{Linear relation/Pencil*}, $\lambda \in  W^{n}(\mathcal{B},\A).$
\end{enumerate}
\end{pf}

\section{Decomposition into operator part and purely multivalued part} \label{Arens_decomposition}
We study the properties of the block numerical range for linear relations with an additional structural property. We use the orthogonal decomposition from~\cite{arens1961operational}. Every closed linear relation $\mathcal{L}$ in a Hilbert space $X$ admits the decomposition  
\begin{align}
\label{Arens decomposition}
	\mathcal{L} = \mathcal{L}_{s} \oplus \mathcal{L}_{\infty}, 
	\quad  \mathcal{L}_\infty := \{0\} \times \mul \mathcal{L},\quad \mathcal{L}_s := \mathcal{L} \ominus \mathcal{L}_{\infty},
\end{align}
into the \emph{operator part } $\mathcal{L}_{s}$ and the \emph{purely multivalued part} $\mathcal{L}_\infty$. In particular, $\mathcal{L}_{s}$ is the graph of an operator with $\dom \mathcal{L}_{s} = \dom \mathcal{L}$ and $\ran \mathcal{L}_{\infty} = \mul \mathcal{L}$. 

It was shown in \cite{arens1961operational} that for any closed
linear relation we have 
$\overline{\dom \mathcal{L}}= (\mul \mathcal{L})^\perp$ 
if and only if $\mul \Ll^*=\mul\Ll$. Moreover, $\ran \Ll_s \subseteq (\mul \mathcal{L})^\perp.$ In this case $\Ll_s$ is the graph of a closed and densely-defined operator in $\overline{\dom\Ll}$ and using the space decomposition 
$$
X=\overline{\dom\Ll}\oplus(\dom\Ll)^\perp
$$ 
we have the following range representation
\begin{align}\label{range representation Arens decomposition}
    \Ll=\ran\begin{bmatrix}
    \mathcal{B}\\
    \A
\end{bmatrix}, \
 \hbox{with} \quad \mathcal{B}=\begin{bmatrix}
      I_{\overline{\dom\Ll}} & 0 \\0&0
\end{bmatrix} \quad \hbox{and}\quad \mathcal{A}=\begin{bmatrix}
     \Ll_s &0\\0& I_{(\dom\Ll)^\perp}
\end{bmatrix}.
\end{align}
The numerical range $W(\Ll)$ was studied in \cite{rofe1990numerical, zhu2025essential} under the condition $\dom \mathcal{L} \subseteq (\mul \mathcal{L})^{\perp}$. 
\begin{theorem}
\label{thm:arens}
    Let $\Ll$ be a closed linear relation in $X$ with $\mul\Ll=\mul\Ll^*$ and range representation given by \eqref{range representation Arens decomposition} and the decomposition $X=X_1\oplus \cdots \oplus X_n\oplus(\dom\Ll)^\perp$. Then the following assertions hold.  
    \begin{enumerate}[(a)]
    \item $\sigma_{p}(\Ll)\subseteq W^{n+1}(\mathcal{B}, \A).$
    \item If $\Ll_s$ is a diagonally dominant $n\times n$ operator matrix of order 0\footnote{This means that, for $\Ll_s=(L_{ij})_{i,j=1}^n$, each off-diagonal block $L_{ij}$ is $L_{jj}$-bounded with relative bound $0$ for $i\neq j$; see \cite{rasulov2018spectral}.}, then $\sigma_{\ap}(\Ll)\subseteq \overline{W^{n+1}(\mathcal{B}, \A)}.$
    \item $W^{n+1}(\mathcal{B}, \A)\setminus \{\infty\}\subseteq W(\mathcal{B}, \A)$,   for all $n\ge  1$.
    \end{enumerate}
\end{theorem}
\begin{pf} 
\begin{enumerate}[(a)]
    \item 
    Let $\lambda \in \sigma_p(\mathcal{L})\setminus \{\infty \}$. 
    Then there exists a nonzero vector $x\in \dom \Ll$ such that $(x, \lambda x)\in \Ll$. Due to (\ref{Arens decomposition}) it follows that there exist $g\in \mul \Ll$ such that $\lambda x= \Ll_s x+g$.
    Since $\mul\Ll=\mul\Ll^*$ yields 
    $\dom \mathcal{L}\subset (\mul \mathcal{L})^\perp$, hence
    $\lambda x= \Ll_s x$ and $g=0$. We easily deduce that $\lambda \in \sigma_p(\Ll_s)$.  By \cite[Theorem 5.1]{rasulov2018spectral}, we obtain that $\sigma_p(\Ll_s)\subseteq W^{n}(\Ll_s)$. By the definition of the block numerical range for the linear relations with range representation \eqref{range representation Arens decomposition}  and for the operator $\Ll_s$, we see that $W^n(\Ll_s)=W^{n+1}(\mathcal{B}, \A)\setminus\{\infty\}$. If $\mul\Ll=\{0\}$ then $\Ll_s=\Ll$ which proves the statement. Otherwise, if $\mul\Ll\neq\{0\}$ then we have that $\infty\in\sigma_p(\Ll)$ and $\infty\in W^{n+1}(\mathcal{B}, \A)$.
    \item Since $\Ll_s$ is a diagonally dominant $n\times n$ operator matrix of order 0, it follows from \cite[Theorem 5.2]{rasulov2018spectral} that $\sigma_{\ap}(\Ll_s)\subseteq \overline{W^{n}(\Ll_s)}.$ From the definition of approximate point spectrum and the orthogonality of the Arens decomposition we have that $\sigma_{\rm ap}(\Ll)\setminus\{\infty\}=\sigma_{\rm ap}(\Ll_s)$.  As established in the proof of (a), $W^n(\Ll_s)=W^{n+1}(\mathcal{B}, \A)\setminus\{\infty\}$. If $\mul\Ll=\{0\}$ then $\Ll_s=\Ll$. Thus, (b) follows. Otherwise, if $\mul\Ll\neq\{0\}$ then we have that $\infty\in\sigma_{p}(\Ll)$ in which case $\infty\in\sigma_{\ap}(\Ll)$  and $\infty\in W^{n+1}(\mathcal{B}, \A)$. 
    \item Let $\lambda \in W^{n+1}(\mathcal{B}, \A)\setminus \{\infty\}$, there exists $f=(f_i)_{i=1}^n\in\prod_{i=1}^n\mathcal{S}_{X_i}$ and choose $f_{n+1}\in(\dom\Ll)^\perp$ and $\|f_{n+1}\|=1$ 
 such that $\lambda \in \sigma(\Ll_f)$ and $\rho(\Ll_f)\neq\emptyset$. Hence, there exists a vector $\alpha=(\alpha_i)_{i=1}^{n+1}\in \mathbb{C}^{n+1}$, $\alpha\neq 0$ with $\|\alpha\|=\sqrt{|\alpha_1|^2 + \cdots + |\alpha_{n+1}|^2}= 1$ such that $(\alpha,\lambda\alpha)\in \Ll_f$. Then, by the range representation \eqref{range representation Arens decomposition}, it follows that $\alpha=\mathcal{B}_f\beta$ and $\lambda\alpha=\A_f\beta$ holds for some $\beta=(\beta_i)_{i=1}^{n+1}\in \mathbb{C}^{n+1}$, $\beta\neq 0$. We have $\lambda=\langle \A_f\beta,\mathcal{B}_f \beta\rangle$. It follows that $\lambda=\langle (\Ll_s)_f\beta', \beta'\rangle$ with $\beta'=(\beta_i)_{i=1}^n$. We have $\alpha=\mathcal{B}_f\beta$. Then, we obtain $1=\|\alpha\|=\|\beta'\|.$ 
 Take the vector $f_{\beta'}:=\sum_{i=1}^n\beta_i f_i$, \begin{align*}
         \|f_{\beta'}\|^2 =\sum _{i=1} ^{n} \|\beta_i f_{i}\| ^2 =\sum _{i=1} ^{n} |\beta_i |^2 = 1.
    \end{align*}
and hence $\lambda=\langle (\Ll_s)_f\beta', \beta'\rangle= \langle \Ll_s f_{\beta'}, f_{\beta'}\rangle $, $(f_{\beta'},\Ll_s f_{\beta'})\in \operatorname{gr}\Ll_s$. Thus, $\lambda \in  W(\Ll_s)$. By the range representation \eqref{range representation Arens decomposition} together with \eqref{Numerical range_Linear relation}, we obtain $W(\Ll_s)=W(\mathcal{B}, \A)$. Consequently, $W^{n+1}(\mathcal{B}, \A)\setminus \{\infty\}\subseteq W(\mathcal{B}, \A)$,   for all $n\ge  1$.
 \end{enumerate}
\end{pf}
\begin{theorem}\label{Arens_Equality}
     Let $\Ll$ be a closed linear relation in $X$ with $\mul\Ll=\mul\Ll^*$ and range 
     representation \eqref{range representation Arens decomposition} and the decomposition $X=\overline{\dom\Ll}\oplus(\dom\Ll)^\perp$. Then 
     $$
     W^2(\mathcal{B}, \A)\backslash \{\infty\}=W(\Ll).
     $$
\end{theorem}
\begin{pf} 
By Definition~\ref{Definition: Block numerical range}, we have
$$ 
W^2(\mathcal{B}, \A)=\bigcup\limits_{\substack{f=(f_1,f_2) \in 
\overline{\dom\Ll}\oplus(\dom\Ll)^\perp , \\
\|f_1\|=\|f_2\|=1, \ 
\rho(\mathcal{L}_f)\neq\emptyset}} \sigma(\mathcal{L}_f).
$$
It follows from the range representation \eqref{range representation Arens decomposition} and \eqref{eq:truncated_relation} that  
\begin{eqnarray*}
W^2(\mathcal{B}, \A) &=& \bigcup\limits_{\substack{\|f_1\|=1, \\ \rho(\mathcal{L}_f)\neq\emptyset}} \sigma\left(\left\{
\left(
\begin{bmatrix}
1 & 0\\
0 & 0
\end{bmatrix}
\begin{bmatrix}
z_1\\
z_2
\end{bmatrix},
\begin{bmatrix}
\langle \Ll_sf_1,f_1\rangle & 0\\
0 & 1
\end{bmatrix}
\begin{bmatrix}
z_1\\
z_2
\end{bmatrix}
\right)
:\ z_1,z_2\in\mathbb C
\right\}
\right)\\ &=& \bigcup\limits_{\substack{\|f_1\|=1, \\ \rho(\mathcal{L}_f)\neq\emptyset}} \sigma\left(\left\{
\left(
(z_1, 0),
(\langle \Ll_sf_1,f_1\rangle z_1, z_2)
\right)
:\ z_1,z_2\in\mathbb C
\right\}
\right). \end{eqnarray*}
Moreover, for $\lambda \in \C$, we obtain $\Ll_f-\lambda = \left\{
\left(
(z_1, 0),
((\langle \Ll_sf_1,f_1\rangle-\lambda) z_1, z_2)
\right)
:\ z_1,z_2\in\mathbb C
\right\}.$ For $\lambda=\langle \Ll_sf_1,f_1\rangle$, we have  $$\Ll_f-\lambda= \left\{
\left(
(z_1, 0),
(0, z_2)
\right)
:\ z_1,z_2\in\mathbb C
\right\}.$$
Since $z_2$ is arbitrary, it follows that $\sigma( \Ll_f)=\{\langle \Ll_sf_1,f_1\rangle\}\cup \{\infty\}$ and $\mul\Ll=\Span\left\{\left[\begin{smallmatrix}
    0\\1
\end{smallmatrix}\right]\right\}.$  Hence, $$ W^2(\mathcal{B}, \A)\backslash \{\infty\}=W(\Ll_s).$$  It is shown in \cite[Proposition 3.1]{zhu2025essential} that $W(\Ll_s) =W(\Ll)$. Consequently, $ W^2(\mathcal{B}, \A)\backslash \{\infty\}=W(\Ll).$\end{pf}
\begin{rmk}  
    The assumption $\mul\Ll=\mul\Ll^*$ is fulfilled for large subclasses of linear relations like self-adjoint linear relations, but it also holds if $\mathcal{L}$ is \emph{maximally dissipative}, see \cite[Lemma 1.6.2]{zbMATH07170295}. 
\end{rmk}

\section{Range representation based on the resolvent}
\label{sec:resolvent-block-numerical-range}

In this section, we study a particular range representation, which requires only $\rho(\Ll)\neq \emptyset.$

Let $X$ be a complex Hilbert space with a finite orthogonal decomposition
$X=X_1\oplus\cdots\oplus X_n,$ where $X_j\neq\{0\}$ for every $j=1,\ldots,n$.  Let $\Ll$ be a closed linear relation in $X$ such that
$\rho(\Ll)\neq\emptyset$, and fix $\mu\in\rho(\Ll)$ with $\mu\neq\infty$. We set
$$
R_\mu:=(\Ll-\mu)^{-1}.
$$
Then $\Ll$ admits the range representation (cf. Proposition \ref{prop:relationfrompsres})
\begin{align}
\label{eq:ran_with_resolvents2}
\Ll
=
\operatorname{ran}
\begin{bmatrix}
R_\mu\\
I+\mu R_\mu
\end{bmatrix}
=
\left\{
\left(
R_\mu z,
z+\mu R_\mu z
\right):
z\in X
\right\}.
\end{align}
Let $f\in\prod_{i=1}^n \mathcal{S}_{X_i}$ then the corresponding compression of the resolvent and the relation are given by
$$
\Ll_{\mu,f}
=
\operatorname{ran}
\begin{bmatrix}
R_{\mu,f}\\
I_n+\mu R_{\mu,f}
\end{bmatrix}
\subseteq
\mathbb{C}^n\times\mathbb{C}^n,\quad R_{\mu,f}
:=
F_f^*R_\mu F_f
\in\mathbb{C}^{n\times n},
$$
where we use the subscript $\mu$ to indicate the dependence of the truncation of the particular range representation. 

The next lemma shows that for linear relations with non-empty resolvent set then  any truncated resolvent representation leads to a linear relation with finite spectrum. 
\begin{lemma}
\label{lem:spectrum-compressed-relation}
Let $\mathcal{L}$ be a linear relation in $X$ with $\mu\in\rho(\mathcal{L})\neq\emptyset$. Then for every
$f\in\prod_{j=1}^n\mathcal{S}_{X_j}$, one has
$$
(\Ll_{\mu,f}-\mu)^{-1}
=
\operatorname{gr}R_{\mu,f}.
$$
In particular,
$
\mu\in\rho(\Ll_{\mu,f}).
$
Moreover,
$$
\sigma(\Ll_{\mu,f})
=
\left\{
\mu+\frac{1}{\nu}:
\nu\in\sigma(R_{\mu,f})
\right\}.
$$
Consequently, $\sigma(\Ll_{\mu,f})$ contains at most $n$ points. Moreover, $0\in\sigma(R_{\mu,f})$ corresponds to $\infty \in \sigma(\Ll_{\mu,f}).$
\end{lemma}
\begin{pf}
The first statement is obvious.
Since $R_{\mu,f}$ is an everywhere-defined bounded operator on
$\mathbb{C}^n$, it follows that $\mu\in\rho(\Ll_{\mu,f})$. Thus $\lambda\in\sigma(\Ll_{\mu,f})\setminus\{\infty\},$ $\lambda \neq \mu$ if and only if there
exists $\nu\in\sigma(R_{\mu,f})$ such that
$$
1-(\lambda-\mu)\nu=0.
$$
Such an equality is possible only if $\nu\neq0$, and in that case $\lambda=\mu+\tfrac{1}{\nu}.$ 
Moreover, $\infty\in\rho(\Ll_{\mu,f})$ if and only if $0\in\rho(R_{\mu,f})$.
\end{pf}

From Lemma~\ref{lem:spectrum-compressed-relation} we obtain that $\mu\in\rho(\Ll_{\mu,f})$ for all $f\in\prod_{i=1}^n\mathcal{S}_{X_i}$ and therefore the assumption $\rho(\Ll_{\mu,f})\neq\emptyset$ in the definition of the block numerical range can be dropped. Hence, the block numerical range is given by 
\[
W_{\mu}^n(\mathcal{L}):=W^n(R_\mu,I+\mu R_{\mu})=\bigcup_{f\in\prod_{j=1}^n\mathcal{S}_{X_j}}
\sigma(\Ll_{\mu,f}).
\]
\begin{proposition}
\label{prop:point-and-approximate-spectrum-enclosure}
Let $\Ll$ be a closed linear relation in a Hilbert space $X$ and $\mu\in\rho(\mathcal{L})\neq\emptyset$ and range representation \eqref{eq:ran_with_resolvents2}. Then the following hold
\begin{enumerate}[(a)]
    \item $W_{\mu}^n(\Ll)
=
\left\{
\mu+\frac{1}{\nu}:
\nu\in W^n(R_\mu)
\right\}$; where $\nu=0$ corresponds to $\infty \in W_{\mu}^n(\Ll).$ 
 \item $\sigma_p(\Ll)
\subseteq
W_{\mu}^n(\Ll)$;
    \item  
    $\sigma_{\mathrm{\ap}}(\Ll)
\subseteq
\overline{W_{\mu}^n(\Ll)}$;
\end{enumerate}
\end{proposition}
\begin{pf}
\begin{enumerate}[(a)]
\item 
Follows from Lemma~\ref{lem:spectrum-compressed-relation} and the definition of the block numerical range \eqref{eq:block_range_matrix}.
\item Let
$\lambda\in\sigma_p(\Ll)$. Then there exists $x\neq0$ such that $(x,\lambda x)\in\Ll.$ 
Since $\mu\in\rho(\Ll)$, we have $\lambda\neq\mu$. Set
$$
\nu:=\tfrac{1}{\lambda-\mu}.
$$
The relation $
(x,\lambda x)\in\Ll
$ implies
$
\bigl((\lambda-\mu)x,x\bigr)
\in
(\Ll-\mu)^{-1}
=
\operatorname{gr}R_\mu.
$
Hence
$$
R_\mu\bigl((\lambda-\mu)x\bigr)=x,
$$
or, equivalently,
$$
R_\mu x=\nu x.
$$
Let $P_{X_j}$ denote the orthogonal projection onto $X_j$. For each
$j=1,\ldots,n$, choose
$$
f_j
=
\frac{P_{X_j}x}{\|P_{X_j}x\|},
$$
if $P_{X_j}x\neq0$, and choose an arbitrary
$f_j\in\mathcal{S}_{X_j}$ if $P_{X_j}x=0$. Set $f=(f_1,\ldots, f_n) \in \Ss_{X_1}\times \ldots \Ss_{X_n}.$ Define
$$
c_j:=\|P_{X_j}x\|,
\qquad
c:=(c_1,\ldots,c_n)^\top.
$$
Then $c\neq0$ and $
x=F_fc.$ Consequently,
$$
\begin{aligned}
R_{\mu,f}c=
F_f^*R_\mu F_fc=
F_f^*R_\mu x=
\nu F_f^*x=
\nu c.
\end{aligned}
$$
Thus, $\nu\in\sigma(R_{\mu,f})\setminus\{0\}.$ 
By Lemma~\ref{lem:spectrum-compressed-relation},
$$\lambda
=
\mu+\tfrac{1}{\nu}
\in
\sigma(\Ll_{\mu,f})
\subseteq
W_{\mu}^n(\Ll).
$$
This proves
$\sigma_p(\Ll)
\subseteq
W_{\mu}^n(\Ll).$ If $\infty \in \sigma_p(\Ll),$ then $\mul\Ll \neq \{0\}.$ Therefore $\ker R_\mu \neq \{0\}$ and hence $0\in \sigma_p(R_\mu).$ By the spectral inclusion for bounded operators, $0\in W^n(R_\mu)$ and from (a) we obtain $\infty \in W_{\mu}^n(\Ll).$ 

\item Let
$\lambda\in\sigma_{\ap}(\Ll)\setminus\{\infty\}$. Then there exists a sequence
$(x_k,y_k)\in\Ll$ such that $\|x_k\|=1$
and $z_k:=y_k-\lambda x_k\rightarrow0.$ Again, $\lambda\neq\mu$, and we again set $
\nu:=\tfrac{1}{\lambda-\mu}.$
Since $(x_k,y_k)\in\Ll$, we have
$
R_\mu(y_k-\mu x_k)=x_k.
$
Using
$
y_k-\mu x_k
=
(\lambda-\mu)x_k+z_k,
$
we obtain
$$
x_k
=
(\lambda-\mu)R_\mu x_k+R_\mu z_k.
$$
Multiplication by $\nu$ gives
\begin{align}
\label{eq:Rmu_to_zero}
(R_\mu-\nu I)x_k
=
-\nu R_\mu z_k
\rightarrow0.
\end{align}

For every $k$, choose
$
f^{(k)}
=
\left(
f_1^{(k)},\ldots,f_n^{(k)}
\right)
\in
\prod_{j=1}^n\mathcal{S}_{X_j}
$ from the normalized block components of $x_k$ as above.
Define
$$
c_j^{(k)}
:=
\|P_{X_j}x_k\|,
\qquad
c^{(k)}
:=
\left(
c_1^{(k)},\ldots,c_n^{(k)}
\right)^\top.
$$
Then $x_k=F_{f^{(k)}}c^{(k)}$ and $\|c^{(k)}\|=\|x_k\|=1.$
Set
$$
A_k
:=
R_{\mu,f^{(k)}}
=
F_{f^{(k)}}^*R_\mu F_{f^{(k)}}.
$$
Then \eqref{eq:Rmu_to_zero} implies that
$$
\begin{aligned}
(A_k-\nu I_n)c^{(k)}
&=
F_{f^{(k)}}^*(R_\mu-\nu I)x_k
\longrightarrow0.
\end{aligned}
$$

Let $d_k:=
\operatorname{dist}\bigl(\nu,\sigma(A_k)\bigr).$
Then, since $A_k$ is an $n\times n$ matrix,
$$
d_k^n
\le \prod_{\lambda\in\sigma(A_k)} |\lambda-\nu|=
\left|\det(A_k-\nu I_n)\right|.
$$
By applying the singular value decomposition argument used in the proof of Lemma \ref{Equivalence: linear relation_Pencil} to the operator $A_k-\nu I_n$, we obtain, for every unit 
vector $c\in\mathbb{C}^n$, 
\begin{align*}
\left|\det(A_k-\nu I_n)\right| &\le
\|(A_k-\nu I_n)c\|
\|A_k-\nu I_n\|^{n-1}.
\end{align*}
Applying this estimate with $c=c^{(k)}$ gives
$$
d_k^n
\le
\|(A_k-\nu I_n)c^{(k)}\|
\|A_k-\nu I_n\|^{n-1}.
$$
Since $
\|A_k-\nu I_n\|
\le
\|R_\mu\|+|\nu|,
$ we conclude that $
d_k\rightarrow0.$ Choose $\nu_k\in\sigma(A_k)$ such that $|\nu_k-\nu|=\operatorname{dist}\bigl(\nu,\sigma(A_k)\bigr)=d_k.$
Then
$
\nu_k\rightarrow\nu.
$
Since $\nu\neq0$, we have $\nu_k\neq0$ for all sufficiently large $k$. For
such $k$, define
$$
\lambda_k
:=
\mu+\frac{1}{\nu_k}.
$$
By Lemma~\ref{lem:spectrum-compressed-relation},
$
\lambda_k
\in
\sigma(\Ll_{\mu,f^{(k)}})
\subseteq
W_{\mu}^n(\Ll).
$ Furthermore,
$
\lambda_k
\rightarrow
\mu+\tfrac{1}{\nu}
=
\lambda.$ 
Therefore, $\lambda\in\overline{W_{\mu}^n(\Ll)}.$ 
This proves $\sigma_{\ap}(\Ll)\setminus\{\infty\}
\subseteq\overline{W_{\mu}^n(\Ll)}.$ We have $\infty\in\sigma_{\ap}(\Ll)$ if and only if $ 0\in\sigma_{\ap}(\Ll^{-1})$. Hence, there exists a sequence $(y_k,x_k)\in \Ll^{-1}$ such that $\|y_k\|=1$ and $\|x_k\|\rightarrow 0$. We have $(y_k,x_k)\in \Ll^{-1} $ if and only if $(x_k,y_k) \in \Ll$, which is equivalent to $ (x_k, y_k-\mu x_k)\in \Ll-\mu$, and this is equivalent to $(y_k-\mu x_k, x_k)\in (\Ll-\mu)^{-1}=\operatorname{gr} R_\mu.$ Therefore, $R_\mu y_k=(I+\mu R_\mu)x_k$. Since $\|x_k\|\rightarrow 0$ and $(I+\mu R_\mu)$ is a bounded operator, it follows that $\|R_\mu y_k\|\rightarrow 0, k\rightarrow \infty.$ Consequently, $0\in\sigma_{\ap}(R_\mu)$ since $\|y_k\|=1$.  \end{enumerate}
\end{pf}
Our second main result is the following.
\begin{theorem}
\label{thm: spectrum-enclosure}
Let $\Ll$ be a closed linear relation in a Hilbert space $X$ and $\mu\in\rho(\mathcal{L})\neq\emptyset$ and range representation \eqref{eq:ran_with_resolvents2}. Then the following holds
$$\sigma(\Ll)
\subseteq
\overline{W_{\mu}^n(\Ll)}.$$
\end{theorem}
\begin{pf} 
For $\lambda\neq\mu$, the resolvent representation of $\Ll$ gives
$$
\lambda\in\rho(\Ll)
\ \text{if and only if } \
I-(\lambda-\mu)R_\mu
\text{ is invertible}.
$$
Consequently,
$$
\lambda\in\sigma(\Ll) \ \text{if and only if } \
\frac{1}{\lambda-\mu}\in\sigma(R_\mu).
$$

Let $\lambda\in\sigma(\Ll)\setminus\{\infty\}$ and set
$
\nu:=\tfrac{1}{\lambda-\mu}.
$
 Then $\nu\in\sigma(R_\mu)\setminus\{0\}.$ 
If $\nu\in\sigma_{\ap}(R_\mu)$, the compression argument from the
proof of part (c) of Proposition \ref{prop:point-and-approximate-spectrum-enclosure} shows that $\nu\in\overline{W^n(R_\mu)}.$

Suppose instead that $\nu\notin\sigma_{\ap}(R_\mu). $
Then $R_\mu-\nu I$ is bounded below and therefore has closed range. Since
$\nu\in\sigma(R_\mu)$, the operator $R_\mu-\nu I$ is not surjective. Hence,
$$
\ker(R_\mu^*-\overline{\nu}I)\neq\{0\}.
$$
Applying the point-spectrum compression argument to $R_\mu^*$ yields $
\overline{\nu}\in W^n(R_\mu^*)$. For every $f$, one has $$F_f^*R_\mu^*F_f=
(F_f^*R_\mu F_f)^*.$$ Therefore, $W^n(R_\mu^*)=
\left\{
\overline{z}:z\in W^n(R_\mu)
\right\},$
and hence $\nu\in W^n(R_\mu).$

We have therefore shown that $
\nu\in\overline{W^n(R_\mu)}.$ 
Since $\nu\neq0$ and the map $z\mapsto\mu+\frac{1}{z}$ 
is continuous in a neighborhood of $\nu$, it follows from Proposition \ref{prop:point-and-approximate-spectrum-enclosure} (a) that
$
\lambda
=
\mu+\frac{1}{\nu}
\in
\overline{W_{\mu}^n(\Ll)}.
$
If $\infty\in\sigma(\mathcal{L})$, then either $\mul\mathcal{L}\neq\{0\}$ or $\dom\Ll\neq X$. Then either $R_{\mu}$ is not surjective or $R_{\mu}$ is not injective which implies that $0\in\sigma(R_{\mu})$. Using the spectral inclusion for bounded operators we get $0\in \overline{W^n(R_{\mu})}$. By Proposition \ref{prop:point-and-approximate-spectrum-enclosure} (a), we obtain $\infty\in \overline{W_{\mu}^n(\Ll)}$. 
\end{pf}
\begin{theorem} 
Let $\Ll$ be a closed linear relation in a Hilbert space $X$ and $\mu\in\rho(\mathcal{L})\neq\emptyset$ and range representation \eqref{eq:ran_with_resolvents2}.  Then the following hold
\begin{enumerate}[(a)]
    \item
If $\mu\in\mathbb{R}$ then $\mathcal{L}$ is symmetric if and only if $R_\mu=R_\mu^*$. 
    \item If $\mathcal{L}$ is symmetric then 
$W_{\mu}^n(\Ll)\subseteq\mathbb{R}\cup\{\infty\}$.
\end{enumerate}

\end{theorem}
\begin{pf}
For the proof of (a), recall from \cite{zbMATH07170295} that the adjoint of relation $\Ll$ in range representation~\eqref{eq:ran_with_resolvents2} is given by $\mathcal{L}^*=\ker[(I+\mu R_\mu)^*,-R_\mu^*]$. This, together with the definition of the symmetry $\mathcal{L}$, i.e.\ $\mathcal{L}\subseteq\mathcal{L}^*$, implies $(I+\mu R_\mu)^*R_\mu-R_\mu^*(I+\mu R_\mu)=0$. Conversely, if $R_{\mu}=R_{\mu}^*$ then $\Ll=\ran\left[\begin{smallmatrix}
    R_\mu \\ I+\mu R_\mu
\end{smallmatrix}\right] \subseteq \ker[(I+\mu R_\mu)^*,-R_\mu^*]=\Ll^*$ implying that $\mathcal{L}$ is symmetric by definition. For the proof of (b), we use  (a), which gives that the truncated linear relation $\Ll_{\mu,f}$ is symmetric and therefore $\sigma(\Ll_{\mu,f})\subseteq\mathbb{R}\cup\{\infty\}$ for all $f\in\prod_{i=1}^n\mathcal{S}_{X_i}$, which implies $W_{\mu}^n(\Ll)\subseteq\mathbb{R}\cup\{\infty\}$.
\end{pf}
Since the spectral enclosures in Proposition~\ref{prop:point-and-approximate-spectrum-enclosure} hold for all resolvent points $\mu$, the inclusion remains true for intersections of these sets for different resolvent points.
\begin{corollary}
    Let $\mathcal{L}$ be a linear relation with $\rho(\mathcal{L})\neq\emptyset$ and let $M\subseteq\rho(\mathcal{L})$. Then 
    \[
  \sigma_p(\mathcal{L})\subseteq\bigcap_{\mu\in M}
W_{\mu}^n(\Ll),\quad  \sigma_{\ap}(\mathcal{L})\subseteq\bigcap_{\mu\in M}
\overline{W_{\mu}^n(\Ll)} ,\quad  \sigma(\mathcal{L})\subseteq\bigcap_{\mu\in M}
\overline{W_{\mu}^n(\Ll)}.
    \]
\end{corollary}

\section{Comparison of the block numerical ranges for an example}\label{example_comparison}
Consider
\[
 \Ll=\ran\begin{bmatrix}\mathcal B\\\mathcal A\end{bmatrix}
 \subseteq\C^6\times\C^6,
 \qquad
 \mathcal B=\diag(1,1,1,1,0,0),
\]
where
\begin{equation*}
\resizebox{0.98\linewidth}{!}{$
\mathcal A=
\begin{bmatrix}
 1.2933-0.1022\mathrm i&0.0809+0.5067\mathrm i&1.1684-0.0900\mathrm i&2.0032-0.1414\mathrm i&0&0\\
 -1.1260+2.2240\mathrm i&-3.0487-0.1532\mathrm i&0.6121+1.8473\mathrm i&-2.2940-0.2141\mathrm i&0&0\\
 0.3502+3.9918\mathrm i&-3.0177+2.1657\mathrm i&-0.3542+0.9676\mathrm i&-3.9118-0.1468\mathrm i&0&0\\
 -0.9006-1.4897\mathrm i&-0.1194-2.1177\mathrm i&-2.1550-0.0772\mathrm i&2.0095-0.6121\mathrm i&0&0\\
 -1.3626-0.4304\mathrm i&-0.4413+0.4890\mathrm i&-0.7631+0.2246\mathrm i&-1.2121-0.1787\mathrm i&0.0381-0.0291\mathrm i&3.5264+0.4325\mathrm i\\
 -1.0129+0.3074\mathrm i&-0.7246-0.6627\mathrm i&0.5736+0.5019\mathrm i&-0.1551+0.3735\mathrm i&0&2.2690+0.3074\mathrm i
\end{bmatrix}.$}
\end{equation*}
Relative to $\C^6=\C^4\oplus\C^2$, write
\[
 \mathcal A=\begin{bmatrix}A_{11}&0\\A_{21}&A_{22}\end{bmatrix}.
\]
The matrix $A_{22}$ is invertible. Eliminating its parameter shows that
\begin{equation*}
 \Ll=\operatorname{gr}A_{11}\oplus(\{0\}\times\C^2).
\end{equation*}
Consequently,
\[
 \dom\Ll=\C^4\oplus\{0\},\qquad
 \mul\Ll=\{0\}\oplus\C^2,
\]
and
\begin{equation*}
 W(\Ll)=W(A_{11}).
\end{equation*}
The finite eigenvalues are the eigenvalues of $A_{11}$:
\begin{align*}
 \lambda_1&\approx-2.8001+0.4000\mathrm i,&
 \lambda_2&\approx-1.4000-1.5000\mathrm i,\\
 \lambda_3&\approx 1.2000+1.8001\mathrm i,&
 \lambda_4&\approx 2.9000-0.6000\mathrm i.
\end{align*}
In addition, the regular pencil has two eigenvalues at $\infty$, counted
with algebraic multiplicity. Direct recomputation gives
\[
 \|\mathcal A^*\mathcal A-\mathcal A\mathcal A^*\|_2
 \approx85.15998,
 \qquad
 \mu(\mathcal A)\approx1.8898\times10^{-2},
\]
so $\mathcal A$ is invertible and strongly nonnormal.

For the quadratic range,
let
\[
 X_1=\C^4\oplus\{0\},\qquad X_2=\{0\}\oplus\C^2.
\]
The compressed pair has the lower-triangular form
\[
 \mathcal B_f=\begin{bmatrix}1&0\\0&0\end{bmatrix},\qquad
 \mathcal A_f=\begin{bmatrix}
 \langle A_{11}f_1,f_1\rangle&0\\ *&\langle A_{22}f_2,f_2\rangle
 \end{bmatrix}.
\]
For every admissible $f_2$, the only finite eigenvalue is
$\langle A_{11}f_1,f_1\rangle$. Such vectors exist, for example the first
canonical vector in the algebraic block. Hence the quadratic range does not improve
the ordinary range in this example:
\begin{equation*}
 W^2(\mathcal B,\mathcal A)=W(A_{11})\cup\{\infty\}.
\end{equation*}

For the cubic range, refine the dynamic block, not the algebraic block:
\[
 X_1=\Span\{e_1,e_2\},\qquad
 X_2=\Span\{e_3,e_4\},\qquad
 X_3=\Span\{e_5,e_6\}.
\]
The same triangular argument gives the exact identity
\begin{equation*}
 W^3(\mathcal B,\mathcal A)
 =W^2(A_{11})\cup\{\infty\}
 \subseteq W(A_{11})\cup\{\infty\}.
\end{equation*}
The inclusion here follows from the explicit structure, not from
Theorem~\ref{spectral_enclosures_pencils}, because
$0\in W^2(\mathcal B)$. Figure~\ref{fig:direct-ranges} displays the exact
ordinary/quadratic equality. The unit vectors entering the compressed matrices are generated using a Sobol low-discrepancy sequence \cite{JoeKuo2008, Sobol1967}, rather than independent pseudorandom samples, in order to obtain a more uniform coverage of the underlying parameter space.

The latter improves reproducibility
and coverage of the parameter sample, but it remains a point-cloud
approximation; dedicated boundary-seeking algorithms are preferable when an
accurate nonconvex boundary is required \cite{Jacob2025}.

\begin{figure}[pos=t]
 \centering
 \begin{subfigure}[t]{0.48\linewidth}
  \centering
  \includegraphics[width=\linewidth]{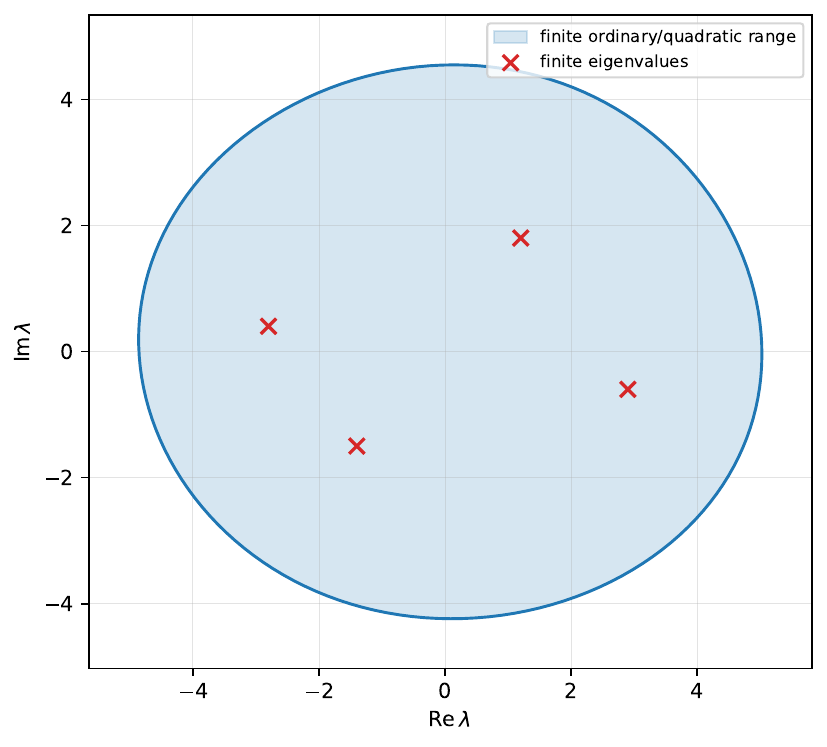}
  \caption{The exact equality
  $W(\Ll)=W^2(\mathcal B,\mathcal A)\setminus\{\infty\}=W(A_{11})$.}
 \end{subfigure}\hfill
 \begin{subfigure}[t]{0.48\linewidth}
  \centering
  \includegraphics[width=\linewidth]{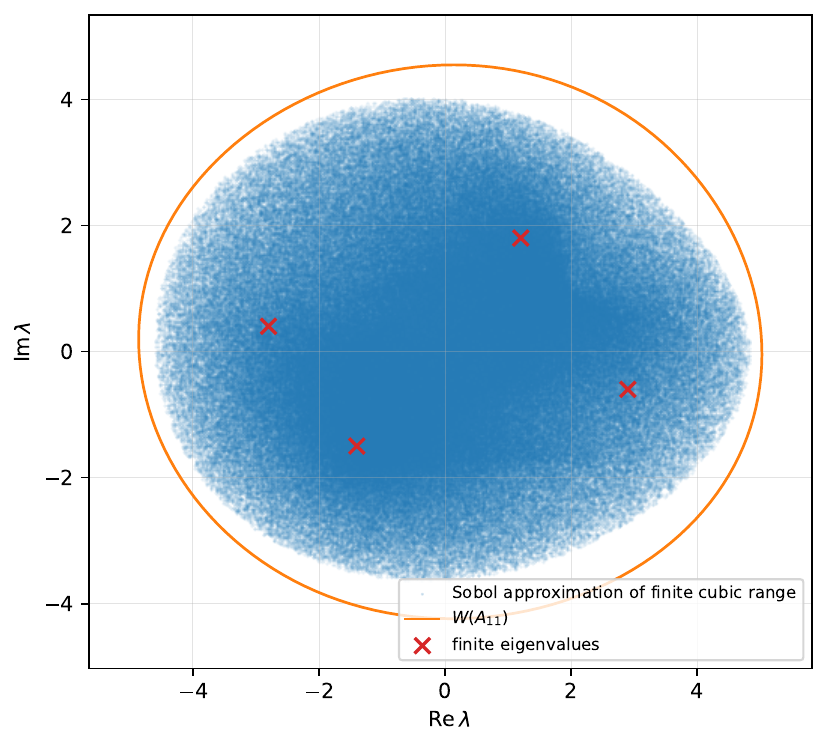}
  \caption{A Sobol approximation of
  $W^3(\mathcal B,\mathcal A)\setminus\{\infty\}=W^2(A_{11})$.}
 \end{subfigure}
 \caption{Direct numerical ranges. Crosses mark the four finite eigenvalues.
 The point $\infty$, present in both relation block ranges, is not displayed.}
 \label{fig:direct-ranges}
\end{figure}

The resolvent ranges admit an exact reduction for the aligned decomposition
$Y_1\oplus Y_2$. For $\mu\in\rho(\Ll)$,
\begin{equation*}
 R_\mu=(\Ll-\mu)^{-1}
 =\begin{bmatrix}(A_{11}-\mu I_4)^{-1}&0\\0&0\end{bmatrix}.
\end{equation*}
Therefore,
\begin{equation}
\label{eq:example-resolvent-range-exact}
 W_{\mu}^2(\Ll)
 =\left\{\mu+\frac1\nu:
 \nu\in W((A_{11}-\mu I_4)^{-1})\setminus\{0\}\right\}
 \cup\{\infty\}.
\end{equation}
Thus the finite boundary can be computed from the support-function boundary of
an ordinary matrix numerical range, followed by the reciprocal Möbius map; no
random sampling is needed.

We use
\[
 \mu_1=-5-4\mathrm i,\quad
 \mu_2=5+4\mathrm i,\quad
 \mu_3=-5+4\mathrm i,\quad
 \mu_4=5-4\mathrm i.
\]
The smallest singular values of $\mathcal A-\mu_j\mathcal B$ are,
respectively,
\[
 2.45122\times10^{-2},\quad
 2.52311\times10^{-2},\quad
 2.52214\times10^{-2},\quad
 2.48189\times10^{-2},
\]
which verifies that all four points lie in the resolvent. In each case,
$0\notin W((A_{11}-\mu_jI)^{-1})$. The computed distances of zero from
these four numerical ranges are approximately
$0.04620$, $0.04284$, $0.04425$, and $0.04437$, respectively, so the finite
part in \eqref{eq:example-resolvent-range-exact} is bounded. Figure
\ref{fig:resolvent-ranges} shows the four sets and their finite-plane
intersection. Using $4096$ support directions, the polygonal area decreases
from approximately $68.11$ for $W(A_{11})$ to $30.06$ for the intersection,
while all four finite eigenvalues remain enclosed. These area values are
numerical diagnostics, not exact invariants.

\begin{figure}[pos=t]
 \centering
 \begin{subfigure}[t]{0.48\linewidth}
  \centering
  \includegraphics[width=\linewidth]{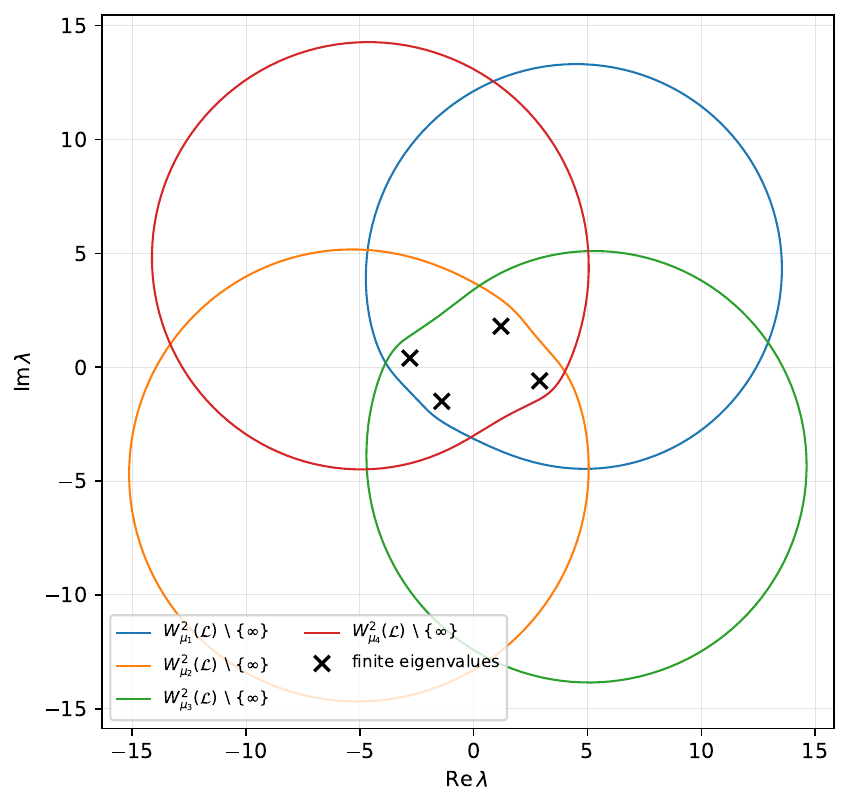}
  \caption{Finite boundaries of the four resolvent block ranges.}
 \end{subfigure}\hfill
 \begin{subfigure}[t]{0.48\linewidth}
  \centering
  \includegraphics[width=\linewidth]{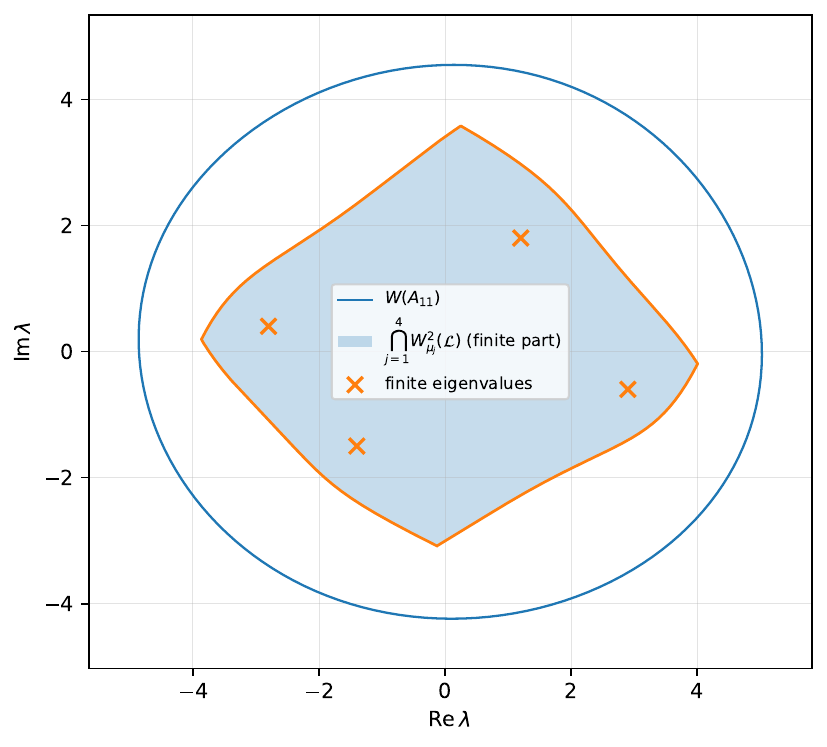}
  \caption{Their finite-plane intersection compared with $W(A_{11})$.}
 \end{subfigure}
 \caption{Resolvent-based enclosures for the aligned decomposition
 $\C^6=\C^4\oplus\C^2$. The common point $\infty$ is not shown.}
 \label{fig:resolvent-ranges}
\end{figure}
All numerical values and figures in this section are reproduced by the script
\path{numerical_experiments.py} supplied with the manuscript.

\section{Application to port-Hamiltonian differential-algebraic equations}
\label{sec:application-ph}

We conclude with an application of the block numerical range to a simple
class of port-Hamiltonian differential-algebraic equations see, e.g., \cite{zbMATH07365308}. Let $X=X_1\oplus X_2$ be a complex Hilbert space and consider
\begin{equation}
\label{eq:ph-dae}
    \frac{\mathrm d}{\mathrm dt}\mathcal Bx(t)
    =
    \mathcal A x(t),
    \qquad
    \mathcal A=\mathcal J-\mathcal R,
\end{equation}
where
\begin{equation}
\label{eq:ph-structure}
    \mathcal B=
    \begin{bmatrix}
        I&0\\
        0&0
    \end{bmatrix},
    \qquad
    \mathcal J=
    \begin{bmatrix}
        J_1&J_{12}\\
        -J_{12}^*&J_2
    \end{bmatrix},
    \qquad
    \mathcal R=
    \begin{bmatrix}
        R_1&0\\
        0&R_2
    \end{bmatrix}.
\end{equation}
We assume that all operators are bounded and
\begin{equation}
\label{eq:ph-assumptions}
    J_1^*=-J_1,
    \qquad
    J_2^*=-J_2,
    \qquad
    R_1=R_1^*\geq0,
    \qquad
    R_2=R_2^*\geq cI
\end{equation}
for some $c>0$. Thus, $\mathcal J^*=-\mathcal J$ and
$\mathcal R=\mathcal R^*\geq0$. In finite dimensions, $J_1$ and $J_2$
are skew-Hermitian matrices (and skew-symmetric matrices in the real
case).

Writing
\begin{equation}
\label{eq:A-blocks-ph}
    \mathcal A
    =
    \begin{bmatrix}
        A_{11}&A_{12}\\
        -A_{12}^*&A_{22}
    \end{bmatrix},
    \qquad
    A_{11}=J_1-R_1,
    \quad
    A_{12}=J_{12},
    \quad
    A_{22}=J_2-R_2,
\end{equation}
we have
$\operatorname{Re}\langle A_{11}x_1,x_1\rangle
    =
    -\langle R_1x_1,x_1\rangle
    \leq0$ and
\begin{equation}
\label{eq:A22-strict-diss}
\operatorname{Re}\langle A_{22}x_2,x_2\rangle
    =
    -\langle R_2x_2,x_2\rangle
    \leq-c\|x_2\|^2.
\end{equation}
Hence $A_{11}$ is dissipative, whereas $A_{22}$ is uniformly
dissipative. 
The natural Hamiltonian associated with \eqref{eq:ph-dae} is
\[
    \mathcal H(x)
    :=
    \frac12\langle \mathcal Bx,x\rangle
    =
    \frac12\|x_1\|^2.
\]
Along sufficiently regular solutions of \eqref{eq:ph-dae},
\begin{align*}
    \frac{\mathrm d}{\mathrm dt}\mathcal H(x(t))
    =
\operatorname{Re}\langle \mathcal A x(t),x(t)\rangle
    =
    -\langle R_1x_1(t),x_1(t)\rangle
    -\langle R_2x_2(t),x_2(t)\rangle
    \leq0.
\end{align*}
Thus \eqref{eq:ph-dae} has the standard dissipative
port-Hamiltonian structure.

We first record some consequences of the strict dissipation in the
algebraic variable.

\begin{proposition}
\label{prop:ph-reduction}
Let $\mathcal B$ and $\mathcal A$ be given by
\eqref{eq:ph-structure}--\eqref{eq:A-blocks-ph} and suppose that
\eqref{eq:ph-assumptions} holds. Then $A_{22}$ is boundedly
invertible and
\begin{equation}\label{rema}
    \|A_{22}^{-1}\|\leq\frac1c.
\end{equation}
The algebraic constraint in \eqref{eq:ph-dae} is therefore equivalent
to
\begin{equation}
\label{eq:ph-constraint}
    x_2=A_{22}^{-1}A_{12}^*x_1,
\end{equation}
and the differential component satisfies
\begin{equation}
\label{eq:reduced-ph}
    \dot x_1=Sx_1,
    \qquad
    S:=
    A_{11}
    +
    A_{12}A_{22}^{-1}A_{12}^*.
\end{equation}
Moreover, $S$ is dissipative. More precisely, if
\begin{equation}
\label{eq:alpha-ph}
    \alpha
    :=
    \sup_{\|x_1\|=1}
    \operatorname{Re}\langle A_{11}x_1,x_1\rangle
    =
    -\inf_{\|x_1\|=1}
    \langle R_1x_1,x_1\rangle
    \leq0,
\end{equation}
then
\begin{equation}
\label{eq:S-halfplane}
    \operatorname{Re}\langle Sx_1,x_1\rangle
    \leq
    \alpha\|x_1\|^2,
    \qquad x_1\in X_1.
\end{equation}
Consequently,
\[
    \sigma(S)
    \subseteq
    \bigl\{
        \lambda\in\mathbb C:
        \operatorname{Re}\lambda\leq\alpha
    \bigr\}.
\]
\end{proposition}

\begin{pf}
By \eqref{eq:A22-strict-diss},
\[
    c\|x_2\|^2
    \leq
    -\operatorname{Re}\langle A_{22}x_2,x_2\rangle
    \leq
    \|A_{22}x_2\|\,\|x_2\|.
\]
Hence
\[
    \|A_{22}x_2\|\geq c\|x_2\|,
\]
so that $A_{22}$ is injective and has closed range. The same argument
applied to
\[
    A_{22}^*=-J_2-R_2
\]
gives $\ker A_{22}^*=\{0\}$. Therefore,
\[
    \overline{\ran A_{22}}
    =
    (\ker A_{22}^*)^\perp
    =
    X_2.
\]
Since $\ran A_{22}$ is also closed, $A_{22}$ is surjective and hence
boundedly invertible, with \eqref{rema}.

The second row of \eqref{eq:ph-dae} gives
$0=-A_{12}^*x_1+A_{22}x_2$,
which proves \eqref{eq:ph-constraint}. Substitution into the first row gives \eqref{eq:reduced-ph}.

It remains to prove the dissipativity estimate. Set $y:=A_{12}^*x_1$ and $z:=A_{22}^{-1}y$. Then
\begin{align*}
    \operatorname{Re}
    \langle A_{12}A_{22}^{-1}A_{12}^*x_1,x_1\rangle
    =
    \operatorname{Re}\langle A_{22}^{-1}y,y\rangle
    =
    \operatorname{Re}\langle z,A_{22}z\rangle
    =
    \operatorname{Re}\langle A_{22}z,z\rangle
    =
    -\langle R_2z,z\rangle
    \leq0.
\end{align*}
Together with \eqref{eq:alpha-ph}, this yields
\eqref{eq:S-halfplane}. Since $S$ is bounded, its spectrum is
contained in the closure of its numerical range, which proves the
last assertion.
\end{pf}

The preceding proposition also shows that the finite spectral
behavior of the differential-algebraic equation is governed by the
Schur complement $S$. Indeed,
\[
    \lambda\mathcal B-\mathcal A
    =
    \begin{bmatrix}
        \lambda I-A_{11}&-A_{12}\\
        A_{12}^*&-A_{22}
    \end{bmatrix},
\]
and, since $A_{22}$ is invertible,
\[
    \lambda\mathcal B-\mathcal A
    \text{ is invertible}
    \quad\Longleftrightarrow\quad
    \lambda I-S
    \text{ is invertible}.
\]
Thus the finite spectrum of the associated linear relation satisfies
\begin{equation*}
    \sigma(\mathcal L)\setminus\{\infty\}
    =
    \sigma(S)
    \subseteq
    \{
        \lambda\in\mathbb C:
        \operatorname{Re}\lambda\leq\alpha
    \}.
\end{equation*}
If $X_2\neq\{0\}$, then $\mathcal B$ is singular and, since $A_{22}$
is injective, $\mathcal A(\ker\mathcal B)\neq\{0\}$. Hence
$\infty\in\sigma_p(\mathcal L)$.

We now show that the quadratic numerical range reproduces precisely
the same dissipativity half-plane.

\begin{theorem}
\label{thm:ph-quadratic-range}
Let $\mathcal B$ and $\mathcal A$ satisfy
\eqref{eq:ph-structure}--\eqref{eq:ph-assumptions}, and let $\alpha$
be defined by \eqref{eq:alpha-ph}. Then
\begin{equation}
\label{eq:ph-W2-halfplane}
    W^2(\mathcal B,\mathcal A)\setminus\{\infty\}
    \subseteq
    \bigl\{
        \lambda\in\mathbb C:
        \operatorname{Re}\lambda\leq\alpha
    \bigr\}.
\end{equation}
More explicitly, for
$f_1\in\mathcal S_{X_1}$ and
$f_2\in\mathcal S_{X_2}$, set
\begin{equation}
\label{eq:ph-abd}
    a:=\langle A_{11}f_1,f_1\rangle,
    \qquad
    b:=\langle A_{12}f_2,f_1\rangle,
    \qquad
    d:=\langle A_{22}f_2,f_2\rangle.
\end{equation}
Then $d\neq0$ and the corresponding compressed relation has spectrum
\begin{equation}
\label{eq:compressed-spectrum-ph}
    \sigma(\mathcal L_f)
    =
    \left\{
        a+\frac{|b|^2}{d},
        \infty
    \right\}.
\end{equation}
In particular,
\begin{equation*}
    \operatorname{Re}
    \left(
        a+\frac{|b|^2}{d}
    \right)
    =
    -\langle R_1f_1,f_1\rangle
    -
    \frac{
        |b|^2\langle R_2f_2,f_2\rangle
    }{|d|^2}
    \leq
    \alpha.
\end{equation*}
\end{theorem}

\begin{pf}
For
$f=(f_1,f_2)\in
\mathcal S_{X_1}\times\mathcal S_{X_2}$,
the compressed matrices are
\[
    \mathcal B_f
    =
    \begin{bmatrix}
        1&0\\
        0&0
    \end{bmatrix},\quad 
    \mathcal A_f
    =
    \begin{bmatrix}
        a&b\\
        -\overline b&d
    \end{bmatrix},
\]
where $a,b,d$ are given by \eqref{eq:ph-abd}. By
\eqref{eq:A22-strict-diss},
$\operatorname{Re}d =
    -\langle R_2f_2,f_2\rangle
    \leq-c<0$, and hence $d\neq0$.

For a finite $\lambda$, the characteristic equation of the compressed
pencil is
\begin{equation*}
    0=\det(\lambda\mathcal B_f-\mathcal A_f)=
    \det
    \begin{bmatrix}
        \lambda-a&-b\\
        \overline b&-d
    \end{bmatrix}
    =-(\lambda-a)d+|b|^2.
\end{equation*}
Since $d\neq0$, the compressed pencil is regular and its unique finite
eigenvalue is $\lambda=a+d^{-1}|b|^2$.
Since $\mathcal B_f$ is singular, $\infty$ is the second spectral
point of the compressed relation. This proves
\eqref{eq:compressed-spectrum-ph}.

Finally, 
$\operatorname{Re}a
    =
    -\langle R_1f_1,f_1\rangle$ 
and
\[
    \operatorname{Re}\frac1d
    =
    \frac{\operatorname{Re}d}{|d|^2}
    =
    -
    \frac{\langle R_2f_2,f_2\rangle}{|d|^2}.
\]
Therefore,
\[
    \operatorname{Re}\lambda
    =
    -\langle R_1f_1,f_1\rangle
    -
    \frac{
        |b|^2\langle R_2f_2,f_2\rangle
    }{|d|^2}
    \leq
    \operatorname{Re}a
    \leq
    \alpha,
\]
which proves \eqref{eq:ph-W2-halfplane}.
\end{pf}

We next compare the quadratic range with the scalar range
$W^1(\mathcal B,\mathcal A)$.

\begin{theorem}
\label{thm:ph-scalar-range}
Under the assumptions of Theorem~\ref{thm:ph-quadratic-range},
\begin{equation}
\label{eq:ph-W1-halfplane}
    W^1(\mathcal B,\mathcal A)\setminus\{\infty\}
    \subseteq
    \bigl\{
        \lambda\in\mathbb C:
\operatorname{Re}\lambda\leq\alpha
    \bigr\}.
\end{equation}
Moreover,
\begin{equation}
\label{eq:ph-W2-W1}
    W^2(\mathcal B,\mathcal A)
    \subseteq
    W^1(\mathcal B,\mathcal A).
\end{equation}
Consequently,
\begin{equation*}
    W^2(\mathcal B,\mathcal A)\setminus\{\infty\}
    \subseteq
    W^1(\mathcal B,\mathcal A)\setminus\{\infty\}
    \subseteq
    \bigl\{
        \lambda\in\mathbb C:
        \operatorname{Re}\lambda\leq\alpha
    \bigr\}.
\end{equation*}
\end{theorem}

\begin{pf}
Let $x=(x_1,x_2)\in X$ be a unit vector. Then
$\langle\mathcal Bx,x\rangle
    =
    \|x_1\|^2$. If $x_1\neq0$, the corresponding finite scalar compressed eigenvalue
is
$\lambda
    =
    \tfrac{\langle\mathcal Ax,x\rangle}{\|x_1\|^2}$.
Since $\mathcal J$ is skew-adjoint,
$\operatorname{Re}\langle\mathcal Jx,x\rangle=0$,
and hence
\begin{align*}
\operatorname{Re}\lambda
    =
    -
    \frac{
        \langle R_1x_1,x_1\rangle
        +
        \langle R_2x_2,x_2\rangle
    }{\|x_1\|^2}\leq
    -
    \frac{\langle R_1x_1,x_1\rangle}{\|x_1\|^2}
    \leq
    \alpha.
\end{align*}
If $x_1=0$, then
\[
    \langle\mathcal Bx,x\rangle=0,
    \qquad
    \operatorname{Re}\langle\mathcal Ax,x\rangle
    =
    -\langle R_2x_2,x_2\rangle
    <0.
\]
Thus the corresponding scalar compressed relation is purely
multivalued and contributes the point $\infty$, but no finite
spectral point. This proves \eqref{eq:ph-W1-halfplane}.

It remains to prove \eqref{eq:ph-W2-W1}. Let
$\lambda\in
    W^2(\mathcal B,\mathcal A)\setminus\{\infty\}$.
Then there exist unit vectors $f_1\in X_1$, $f_2\in X_2$ and a
nonzero vector $z=(z_1,z_2)^\top\in\mathbb C^2$ such that
\[
    \mathcal A_fz
    =
    \lambda\mathcal B_fz.
\]
Since $d\neq0$, the second row of this equation implies that
$z_1\neq0$. Set
$x=F_fz=z_1f_1+z_2f_2$.
Taking the inner product of the compressed eigenvalue equation with
$z$ gives
\[
    \langle\mathcal Ax,x\rangle
    =
    \lambda\langle\mathcal Bx,x\rangle
    =
    \lambda|z_1|^2.
\]
Since $z_1\neq0$, after normalizing $x$, this shows
\[
    \lambda=\frac{\langle\mathcal Ax,x\rangle}
         {\langle\mathcal Bx,x\rangle}\in
    W^1(\mathcal B,\mathcal A)\setminus\{\infty\}.
\]
Finally, $\infty$ belongs to both block numerical ranges because the
compressed matrix $\mathcal B_f$ is singular. Hence
\eqref{eq:ph-W2-W1} follows.
\end{pf}

\subsection*{Acknowledgments}
This research was funded by the German Research Foundation (DFG) under project number 567253746.












\bibliographystyle{cas-model2-names}

\bibliography{cas-refs}



\end{document}